\documentclass[11pt]{article}
\usepackage{graphicx,amsmath,amsfonts,amssymb,bbm,amsthm}
\usepackage{psfrag}
\usepackage{epsf}
\usepackage{color}

\makeatletter\@addtoreset {equation}{section}\makeatother

\theoremstyle{plain}
\newtheorem{theo}{Theorem}
\newtheorem{lem}{Lemma}[section]

\newtheorem{prop}{Proposition}[section]

\theoremstyle{remark}
\newtheorem{rem}{Remark}[section]

{\hspace*{\fill}$\rule{.3\baselineskip}{.35\baselineskip}$\end{trivlist}}

\newcommand{\R}{\mathbb{R}}

\newcommand{\Z}{\mathbb{Z}}
\newcommand{\N}{\mathbb{N}}

\renewcommand{\geq}{\geqslant}
\renewcommand{\leq}{\leqslant}

\renewcommand{\phi}{\varphi}
\newcommand{\be}{\begin{eqnarray}}
\newcommand{\ee}{\end{eqnarray}}

\newcommand{\eps}{\varepsilon}

\def\Tend#1#2{\mathop{\longrightarrow}\limits_{#1\rightarrow#2}}

\begin{document}

\title{\bf Justification of the log-KdV equation in granular chains: the case of precompression}

\author{Eric Dumas$^{1}$ and Dmitry Pelinovsky$^{2}$ \\
{\small $^{1}$ Institut Fourier, Universit\'e Grenoble 1,
38402 Saint Martin d'H\`eres cedex, France} \\
{\small $^{2}$ Department of Mathematics, McMaster
University, Hamilton, Ontario, Canada, L8S 4K1}  }

\date{\today}
\maketitle

\begin{abstract}
For travelling waves with nonzero boundary conditions, we justify the logarithmic Korteweg--de Vries
equation as the leading approximation of the Fermi-Pasta-Ulam lattice with Hertzian nonlinear potential
in the limit of small anharmonicity. We prove control of the approximation error for the travelling solutions
satisfying differential advance-delay equations, as well as control of the approximation error
for time-dependent solutions to the lattice equations on long but finite time intervals. We also show
nonlinear stability of the travelling waves on long but finite time intervals.
\end{abstract}

\section{Introduction}

Solitary waves in anharmonic granular chains with Hertzian interaction forces are modelled by the
Fermi--Pasta--Ulam (FPU) lattices with non-smooth nonlinear potentials \cite{Nesterenko}.
We write the FPU lattice in the form
\begin{equation}
\label{FPU}
\ddot{u}_n = V_\alpha'(u_{n+1}) - 2 V_\alpha'(u_n) + V_\alpha'(u_{n-1}),
\quad n \in \mathbb{Z},
\end{equation}
where $(u_n)_{n \in \mathbb{Z}}$ is a function of the time
$t\in\mathbb{R}$, with values in $\mathbb{R}^{\mathbb{Z}}$,
and the dot denotes the time derivative. In terms of the FPU lattice,
$u_n$ corresponds to the relative displacement between locations of two adjacent particles.
The Hertzian nonlinear potential $V_{\alpha} \in C^2(\mathbb{R})$ is given by
\begin{equation}
\label{Hertzian}
V_\alpha(u) = \frac{1}{1+\alpha} |u|^{1+\alpha} H(-u), \quad \alpha > 1,
\end{equation}
where $H(u)$ is the standard Heaviside function. Recently, the FPU lattice (\ref{FPU})--(\ref{Hertzian})
in the limit of small anharmonicity of the Hertzian interaction forces (that is, for $\alpha = 1 + \epsilon^2$
with $\epsilon \to 0$) was formally reduced to the logarithmic Korteweg--de Vries (log-KdV) equation \cite{Chat,JP13}:
\begin{equation}
\label{logKdV}
2 v_{\tau} + \frac{1}{12} v_{xxx} + (v \log|v|)_x = 0, \quad (x,\tau) \in \mathbb{R} \times \mathbb{R},
\end{equation}
with the asymptotic correspondence $u_n(t) \approx -v(x,\tau)$, $x = \epsilon(n-t)$, $\tau = \epsilon^3 t$.
Here and in what follows, we denote partial derivatives by subscripts. Note that
in the derivation of (\ref{logKdV}), $v$ is assumed to be positive (otherwise, the Heaviside function
should appear in the nonlinear term). Experimental evidences
for validity of the limit $\alpha \to 1$ in the context of granular chains with hollow particles
can be found in \cite{daraio}.

The log--KdV equation (\ref{logKdV}) has a two-parameter family of Gaussian travelling waves
\begin{equation}
\label{soliton-orbit}
v(x,\tau) = e^{2b} v_G(x- b \tau -a), \quad a,b \in \mathbb{R},
\end{equation}
where $v_G$ is a symmetric standing wave given by
\begin{equation}
\label{Gaussian}
v_G(x) := \sqrt{e} e^{-3 x^2}, \quad x \in \mathbb{R}.
\end{equation}
Global solutions to the log--KdV equation (\ref{logKdV}) in the energy space
\begin{equation}
\label{energy-space}
X := \left\{ v \in H^1(\mathbb{R}) : \quad v^2 \log|v| \in L^1(\mathbb{R}) \right\},
\end{equation}
were constructed in \cite{CP14}. In addition, spectral and linearized stability
of Gaussian travelling waves were proved in \cite{CP14} with analysis of the linearized evolution
problem. Unfortunately, technical difficulties exist to prove nonlinear orbital stability
of Gaussian travelling waves, as well as to construct solutions in spaces of higher regularity \cite{CP14}.
The technical difficulties are caused by the necessity to control the logarithmic nonlinearity
near $v = 0$, where it is not differentiable.

This paper addresses a different problem, namely the rigorous justification of the
log--KdV equation (\ref{logKdV}) in the context of the FPU lattice (\ref{FPU})--(\ref{Hertzian}).
Numerical approximations of time-dependent solutions to the FPU lattice in \cite{JP13} suggest that
the Gaussian travelling waves represent well the stable solitary waves in granular chains, which are known to
propagate robustly in physical experiments \cite{sen}. Therefore,
it becomes relevant to control the approximation error between the corresponding
solutions to the FPU lattice (\ref{FPU})--(\ref{Hertzian}) and the log--KdV equation
(\ref{logKdV}).

In a similar context of FPU lattices with sufficiently smooth nonlinear potential $V$,
small-amplitude solutions are described by the celebrated KdV equation. In a series
of papers, Friesecke and Pego \cite{FP1,FP2,FP3,FP4} justified the KdV approximation for
travelling waves and proved the nonlinear stability of small-amplitude solitary waves
in generic FPU chains from analysis of the orbital and asymptotic stability of KdV solitons.
Later these results were extended to the proof of asymptotic stability of several solitary
waves in the FPU lattices by Mizumachi \cite{miz1,miz2} and Hoffmann and Wayne \cite{HW1,HW2}.
Independently, validity of the KdV equation for time-dependent solutions
on the time scale of $\mathcal{O}(\epsilon^{-3})$
was obtained by Schneider--Wayne \cite{SW00} and Bambusi--Ponno \cite{BP06}.
Recently, these results were generalized for polyatomic FPU lattices in \cite{GMWZ}.
Because of the lack of smoothness of the potential $V_\alpha$
in (\ref{FPU})--(\ref{Hertzian}), and therefore also of
the nonlinearity in the log--KdV equation (\ref{logKdV}) near the origin,
none of the previous results can be applied to the FPU lattice (\ref{FPU})--(\ref{Hertzian}).

As a first step towards the ultimate goal of justification
of the log--KdV equation (\ref{logKdV}), we shall here consider solutions
with nonzero (positive) boundary conditions at infinity. In other words,
we shall consider solutions bounded from below by some positive constant
and satisfying the boundary conditions
$v(x,t) \to v_0 > 0$ as $|x| \to \infty$ for all $t \in \mathbb{R}$.
In the context of the anharmonic granular chains (\ref{FPU})--(\ref{Hertzian}),
these boundary conditions correspond to the constant precompression force applied to the granular chains.

The precompression technique is well-known both numerically and experimentally for regularization
of responses of granular chains \cite{Nesterenko}. Typically, small-amplitude
perturbations of the constant precompressed state are handled through
Taylor expansion of the nonlinearity, thanks to the smoothness of the nonlinear potential $V(u)$
or the logarithmic nonlinearity near any point $v_0 > 0$.
In comparison to this standard technique, we avoid Taylor series
expansion and consider large-amplitude solutions to the FPU lattice (\ref{FPU})--(\ref{Hertzian})
with $\alpha = 1 + \epsilon^2$ in the limit $\epsilon \to 0$. In this
way, we confirm the validity of the log--KdV equation (\ref{logKdV})
with nonzero boundary conditions for the existence and stability of travelling waves.

Note that the nonlinear function $p(v) = v \log|v|$ in the log--KdV equation (\ref{logKdV})
satisfies for any $v \geq v_0 > 0$ the general assumption $p''(v) > 0$ and $p'''(v) \leq 0$
required by the orbital stability theory for large-amplitude travelling waves of the generalized KdV equation
(see Theorem 1 in \cite{H11}).
Consequently, travelling waves of arbitrary amplitudes with the boundary conditions $v \to v_0$ as $x \to
\pm \infty$ are orbitally stable with respect to the time evolution of the log--KdV equation
(\ref{logKdV}) in the classical sense \cite{Pava}.

There are three main results in our work. First, we study travelling wave solutions
to (\ref{FPU})--(\ref{Hertzian}) with $\alpha = 1 + \epsilon^2$, under the form
$u_n(t)=u(n-ct)$ with speed $c = (v_0^{\epsilon^2} (1 + \lambda\epsilon^2))^{1/2}$,
for any $\lambda>1$. We provide a rigorous
approximation of such travelling waves, in the limit $\epsilon\rightarrow0$, by means of travelling solutions
to the log--KdV equation \eqref{logKdV}.

Next, we show that a simple energy argument gives nonlinear stability
of the previously constructed (large-amplitude)  travelling wave solutions
to the FPU lattice equations (\ref{FPU})--(\ref{Hertzian})
with $\alpha = 1 + \epsilon^2$ on the time scale $\mathcal{O}(\epsilon^{-3})$, where
the approximation of the log--KdV equation (\ref{logKdV}) is formally applicable.
The energy argument we develop here does not
use the spectral information on the linearized log--KdV equation and holds
for time-dependent perturbations, which may violate the scaling of space and time variables
resulting in the log--KdV equation (\ref{logKdV}). It only uses the precise
justification result for the travelling waves of the FPU lattice.

Finally, we control the error in the approximation of time-dependent solutions
to the FPU lattice by solutions to the log--KdV equation up to the time scale $\mathcal{O}(\epsilon^{-3})$
by extending the same energy argument used for control of the nonlinear stability of travelling waves.

Although our results are analogous to the outcomes of the corresponding works
\cite{FP1} and \cite{SW00}, a different analytical
technique is adopted to obtain the justification and stability results.
The technique is thought to be applicable to a much large
class of FPU models which result in the generalized KdV equation
with possibly large-amplitude travelling waves.
We also point out that the methods of neither \cite{FP1} nor \cite{SW00} cannot be
immediately applied to the justification of the log--KdV equation (\ref{logKdV})
because they require the smallness of the travelling wave amplitude.

In more details, we use the method of decomposition of solutions
in the Fourier space, which was originally developed in \cite{Schn1}
and used in \cite{DU,IW,Schn2} (see also Chapter 2 in \cite{Pel-book})
for the justification of asymptotic reductions of solitary waves in the nonlinear Schr\"{o}dinger
equation with a periodic potential. This technique is alternative to the method of Friesecke and Pego
\cite{FP1} that relies on approximations of roots of the dispersion relations and on an appropriate
version of a fixed-point theorem. We also use fixed-point arguments but in a more classical way.

While the strategy adopted in \cite{FP2,FP3,FP4} gives
nonlinear stability results for FPU travelling waves globally in time,
it applies only to the small-amplitude travelling waves. It also relies on the spectral
information of the linearized KdV equation, modulation equations along
the two-dimensional manifold of the travelling waves, and careful analysis of linearized
advance-delay equations, all of which may not
be available when dealing with the log--KdV equation (\ref{logKdV}).

The plan of the paper is as follows. Section 2 presents the main results.
Section 3 is devoted to the justification of the log--KdV approximation for the travelling
waves of the FPU lattice. Section 4 is devoted to the orbital stability
of the FPU lattice travelling waves on the time scale $\mathcal{O}(\epsilon^{-3})$.
Section 5 describes justification of the log--KdV equation for the time-dependent
solution to the FPU lattice. Section 6 discusses these results
in the context of general FPU lattices.

\section{Main results}

Substituting the travelling wave ansatz $u_n(t) = u(z)$ with $z = n-ct$ for a positive speed $c > 0$
into the FPU lattice (\ref{FPU})--(\ref{Hertzian}),
we obtain the differential advance-delay equation
\begin{equation}
\label{diffadvdel-eq}
c^2 u''(z) = - \Delta |u|^{\alpha} H(-u)(z), \quad z \in \mathbb{R},
\end{equation}
where $\alpha > 1$, and $\Delta$ is the discrete Laplacian operator on the infinite line,
$$
\Delta f(z) := f(z+1) - 2 f(z) + f(z-1).
$$
Since the limit $\alpha \to 1$ is considered, we set $\alpha := 1 + \epsilon^2$ for a small
positive $\epsilon$. Here and in the sequel, we shall drop the dependence of the functions
(such as $u$) upon $\epsilon$ for simplicity, and only mention this dependence in the
main statements. With a precompression level $v_0 > 0$, we set
\begin{equation}
\label{scaling transformation}
u(z) = -v_0 (1 + w(z)) \quad \mbox{\rm and} \quad
c^2 = v_0^{\epsilon^2} (1 + \mu),
\end{equation}
where $\mu > -1$ is an arbitrary parameter
and $w(z)$ is assumed to decay to zero at infinity and to be bounded in the interval
\begin{equation}
\label{apriori-bound}
-1 < C_- \leq w(z) \leq C_+ < \infty, \quad \mbox{\rm for every} \; z \in \mathbb{R},
\end{equation}
where $C_{\pm}$ are $\epsilon$-independent and $C_+$ does not have to be smaller than one
(that is, $\| w \|_{L^{\infty}}$ may exceed one). Under the a priori bound (\ref{apriori-bound}),
we rewrite the existence problem in the form
\begin{equation}
\label{DFA}
(1+\mu) w''(z) = \Delta \tilde{V}_{\epsilon}(w)(z), \quad z \in \mathbb{R}.
\end{equation}
Here the potential
$$
\tilde{V}_{\epsilon}(w) := \frac{1}{2+\epsilon^2} \left[ (1 + w)^{2+\epsilon^2} - 1 \right] - w, \quad w > -1
$$
is $C^2(-1,\infty)$, positive near $w = 0$, and $\tilde{V}_{\epsilon}(w)/w^2$ increases strictly with $w$ for all $w \in (0,\infty)$.
For such potentials, Theorem 1 of Friesecke and Wattice \cite{friesecke} applies (as it was also noted in \cite{mackay}).
By this theorem, which is proved by a variational method based on the concentration compactness principle,
there exists a nontrivial positive solution $w \in H^1(\mathbb{R})$ of the differential advance-delay
equation (\ref{DFA}) for some parameter $\mu$ satisfying the constraint
$1 + \mu > \tilde{V}_{\epsilon}''(0) = 1 + \epsilon^2$ (that is, for $\mu > \epsilon^2$).
Moreover, recent work \cite{StKev2013} suggests that these travelling waves are smooth and exponentially
localized.

To obtain the formal limit to the stationary log--KdV equation, we
set the variables $x = \epsilon z$ and $W(x) = w(z)$, use
the Taylor expansions
\begin{equation}
\label{expansion-0}
\Delta w(z) = \epsilon^2 W''(x) + \frac{1}{12} \epsilon^4 W''''(x) + \mathcal{O}(\epsilon^6 W^{(6)}(x)),
\end{equation}
and
\begin{equation}
\label{expansion-2}
\tilde{V}_{\epsilon}'(w) =  (1+w)^{1+\epsilon^2} - 1 = w + \epsilon^2 (1+w) \log(1+w) + \mathcal{O}(\epsilon^4 (1+w) \log^2(1+w)),
\end{equation}
and finally integrate \eqref{DFA} with $\mu = \epsilon^2 \lambda$ twice in $x$ subject to the zero
boundary conditions for $W$ and its derivatives. Truncating at the leading order $\mathcal{O}(\epsilon^4)$,
we obtain the stationary log--KdV equation
\begin{equation}
\label{stationaryLogKdV}
\lambda W(x) =  \frac{1}{12} W''(x) + (1+W) \log(1+W), \quad x \in \mathbb{R}.
\end{equation}
By Proposition \ref{proposition-soliton} below, there exists a unique positive and even solution
$W_{\rm stat} \in H^{\infty}(\mathbb{R})$ to the stationary log--KdV equation (\ref{stationaryLogKdV})
with $\lambda > 1$. We are now ready to formulate the main result on the rigorous
justification of this formal approximation.

\begin{theo}
\label{theorem-stationary}
Set $\mu := \epsilon^2 \lambda$ with fixed $\epsilon$-independent parameter $\lambda > 1$.
There exist positive constants $\epsilon_0$ and $C_0$ such that for every $\epsilon \in (0,\epsilon_0)$,
there exists a unique even solution $w_{{\rm stat},\epsilon}$
to the differential advance-delay equation (\ref{DFA}) in $L^2(\mathbb{R}) \cap L^{\infty}(\mathbb{R})$
such that
\begin{equation}
\label{bound-major}
\sup_{z \in \mathbb{R}} | w_{{\rm stat},\epsilon}(z) - W_{\rm stat}(\epsilon z) | \leq C_0 \epsilon^{1/6},
\end{equation}
where $W_{\rm stat}$ is the unique positive and even solution to the stationary log--KdV equation
(\ref{stationaryLogKdV}).
Moreover, $w_{{\rm stat},\epsilon} \in H^{\infty}(\mathbb{R})$ and for every $k \in \mathbb{N}$,
there is a positive $\epsilon$-independent constant $C_k$ such that
\begin{equation}
\label{bound-major-derivatives}
\sup_{z \in \mathbb{R}} | \partial^k_z w_{{\rm stat},\epsilon}(z)
- \epsilon^k \partial_x^k W_{\rm stat}(\epsilon z) | \leq C_k \epsilon^{k+1/6}.
\end{equation}
\end{theo}


\begin{rem}
It follows from analysis of the roots of the dispersion relation
associated with the differential advance-delay equation (\ref{DFA}) 
that $w$ decays to zero exponentially at infinity (see Section 5 in \cite{FP1}).
\end{rem}

Using the scaling transformation
$$
u_n(t) = -v_0 \left( 1 + w_n(t') \right), \quad t' = v_0^{\epsilon^2/2} t,
$$
we can write the FPU lattice in the (formally) equivalent form of the first-order system
\begin{equation} \label{FPU-lattice}
\left\{
\begin{split}
& \dot{w}_n = p_{n+1}-p_n, \\
& \dot{p}_n = \tilde{V}_{\epsilon}'(w_n) - \tilde{V}_{\epsilon}'(w_{n-1}),
\end{split}
\right.
\quad n \in \mathbb{Z}.
\end{equation}
Any $(w,p)\in C^1(\R,l^2(\Z))$ solution to the first-order system \eqref{FPU-lattice},
with $w_n>-1$ for all $n\in\mathbb{Z}$, provides a $C^2(\R,l^2(\Z))$ solution
$u$ to the scalar second-order equation \eqref{FPU}.
The FPU lattice equations (\ref{FPU-lattice}) admit the conserved energy
\begin{equation}
\label{FPU-energy}
H := \frac{1}{2} \sum_{n \in \mathbb{Z}} p_n^2 + \sum_{n \in \mathbb{Z}} \tilde{V}_{\epsilon}(w_n).
\end{equation}
Note that the dot in (\ref{FPU-lattice}) applies with respect to the new variable $t'$. In what follows,
we will use the same notation $t$ for the independent time variable of the FPU system (\ref{FPU-lattice})
for convenience.

Since shift operators are bounded in $l^2(\mathbb{Z})$, it is easy to show the local (in time)
well-posedness of the Cauchy problem associated with the FPU system \eqref{FPU-lattice} in $l^2(\mathbb{Z})^2$.
Furthermore, the energy conservation (\ref{FPU-energy})
and the embedding of $l^2(\mathbb{Z})$ in $l^{\infty}(\mathbb{Z})$ ensures global existence of the solutions,
at least for small initial data. For large initial data, any solution to \eqref{FPU-lattice}
provides a solution to \eqref{FPU} as long as all components of $u$ remain strictly negative, that is, as long as
\begin{equation}
\label{apriori-bound-time}
-1 < C_- \leq w_n(t) \leq C_+ < \infty, \quad \mbox{\rm for every} \; n \in \mathbb{N},
\end{equation}
where $C_{\pm}$ are $\epsilon$ and $t$-independent constants. It may be hard to control this condition
during evolution for general initial data, but our study addresses time-dependent solutions
near the travelling wave of Theorem \ref{theorem-stationary}, which definitely satisfies the
bounds (\ref{apriori-bound-time}). Let us emphasize once again that
the travelling waves and the solutions we consider are not small-amplitude
solutions to the FPU lattice \eqref{FPU-lattice}.

We define a reference travelling wave $(w_{\rm trav},p_{\rm trav}) \in C^1(\R,l^2(\mathbb{Z}))$
solution to the FPU lattice \eqref{FPU-lattice} by
\begin{equation}
\label{solitary-wave-FPU}
(w_{\rm trav})_n(t) = w_{\rm stat}(n - ct), \quad (p_{\rm trav})_n(t) = p_{\rm stat}(n-ct),
\end{equation}
where $c^2 = 1 + \epsilon^2 \lambda$ is the squared wave speed,
$w_{\rm stat}$ is given by Theorem \ref{theorem-stationary},
and $p_{\rm stat}$ is found from the advance equation
$-c w_{\rm stat}'(z) = p_{\rm stat}(z+1) - p_{\rm stat}(z)$.
We now ask if the travelling wave given by (\ref{solitary-wave-FPU})
is stable in the time evolution of the FPU lattice (\ref{FPU-lattice})
with small $\epsilon$ at least on the time scale of $\mathcal{O}(\epsilon^{-3})$,
when the approximation of the log--KdV equation is applicable.

The following theorem gives the affirmative answer to the question of the nonlinear stability
of the FPU travelling waves and specifies the precise conditions, in which the nonlinear stability
of the travelling wave is understood. In particular, this result ensures existence
of the time-dependent solution $(w,p)$ to the FPU lattice (\ref{FPU-lattice})
up to $\mathcal{O}(\epsilon^{-3})$ times.

\begin{theo}
\label{theorem-stability}
As in Theorem \ref{theorem-stationary}, set $\mu := \epsilon^2 \lambda$
with fixed $\epsilon$-independent parameter $\lambda > 1$.
For every $\tau_0 > 0$, there exist positive constants $\epsilon_0$, $\delta_0$ and $C_0$ such that,
for all $\epsilon \in (0,\epsilon_0)$, when initial data $(w_{{\rm ini},\epsilon},p_{{\rm ini},\epsilon})
\in l^2(\mathbb{R})$ satisfy
\begin{equation}
\label{bound-initial}
\delta := \| w_{{\rm ini},\epsilon} - w_{{\rm trav},\epsilon}(0) \|_{l^2} +
\| p_{{\rm ini},\epsilon} - p_{{\rm trav},\epsilon}(0) \|_{l^2} \leq \delta_0,
\end{equation}
then the unique solution $(w_{\epsilon},p_{\epsilon})$
to the FPU lattice equations (\ref{FPU-lattice}) with initial data $(w_{{\rm ini},\epsilon},p_{{\rm ini},\epsilon})$
belongs to $C^1([-\tau_0\epsilon^{-3},\tau_0\epsilon^{-3}],l^2(\mathbb{Z}))$ and satisfies
\begin{equation}
\label{bound-final}
\| w_{\epsilon}(t) - w_{{\rm trav},\epsilon}(t) \|_{l^2} +
\| p_{\epsilon}(t) - p_{{\rm trav},\epsilon}(t) \|_{l^2}
\leq C_0 \delta, \quad t \in \left[-\tau_0 \epsilon^{-3},\tau_0 \epsilon^{-3}\right].
\end{equation}
\end{theo}

\begin{rem}
According to \cite{H11}, the solitary wave $W$ of the stationary log--KdV equation (\ref{stationaryLogKdV}) is orbitally stable
in the time evolution of the log--KdV equation
\begin{equation}
\label{LogKdV-background}
2 W_{\tau} + \frac{1}{12} W_{\xi \xi \xi} + (g(W))_{\xi} = 0, \quad g(W) := (1+W) \log(1+W),
\end{equation}
where $\tau = \epsilon^3 t$ and $\xi = \epsilon (n-t)$ are scaled variables of the FPU lattice (\ref{FPU-lattice}).
From Theorem \ref{theorem-stability} and this orbital stability result,
one can expect that the time-dependent version of the log--KdV
equation (\ref{LogKdV-background}) is a valid approximation of the time-dependent solutions
to the FPU lattice (\ref{FPU-lattice}) modulated on the spatial scale $\mathcal{O}(\epsilon^{-1})$
up to the time scale of $\mathcal{O}(\epsilon^{-3})$.
\end{rem}

\begin{rem}
Compared to the log--KdV equation (\ref{LogKdV-background}),  Theorem \ref{theorem-stability} gives
also stability of the FPU travelling waves with respect to modulations on any other spatial scale,
nevertheless, up to the time scale of $\mathcal{O}(\epsilon^{-3})$ only.
\end{rem}

Finally, we justify the approximation of time-dependent solutions
to the FPU lattice (\ref{FPU-lattice}) by the log--KdV equation \eqref{LogKdV-background}.
Technically, when a solution $W$ to \eqref{LogKdV-background} is given, we define
\begin{equation} \label{approxmomentum}
P_\epsilon := -W + \frac{\epsilon}{2} W_{\xi} - \frac{\epsilon^2}{8} W_{\xi \xi}
- \frac{\epsilon^2}{2} g(W) + \frac{\epsilon^3}{48} W_{\xi \xi \xi}
+ \frac{\epsilon^3}{4} (g(W))_{\xi},
\end{equation}
so that $(W,P_\epsilon)$ solves the first equation in \eqref{FPU-lattice} up to $\mathcal{O}(\epsilon^4)$ terms.
The following theorem controls the approximation error up to $\mathcal{O}(\epsilon^{-3})$ times.

\begin{theo}
\label{theorem-justification}
Let $W \in C([-\tau_0,\tau_1],H^s(\mathbb{R}))$ be a solution to the log--KdV equation
\eqref{LogKdV-background} for some integer $s \geq 6$ and some $\tau_0,\tau_1\geq0$.
Assume that there exists $r_W>-1$ such that $W\geq r_W$.
Then there exist positive constants $\epsilon_0$ and $C_0$ such that,
for all $\epsilon \in (0,\epsilon_0)$, when initial data $(w_{{\rm ini},\epsilon},p_{{\rm ini},\epsilon})
\in l^2(\mathbb{R})$ are given such that
\begin{equation}
\label{bound-initial-time}
\| w_{{\rm ini},\epsilon} - W(\epsilon \cdot,0) \|_{l^2} +
\| p_{{\rm ini},\epsilon} - P_\epsilon(\epsilon \cdot,0) \|_{l^2} \leq \epsilon^{3/2},
\end{equation}
with $P_\epsilon$ given by \eqref{approxmomentum},
the unique solution $(w_{\epsilon},p_{\epsilon})$
to the FPU lattice equations (\ref{FPU-lattice}) with initial data $(w_{{\rm ini},\epsilon},p_{{\rm ini},\epsilon})$
belongs to $C^1([-\tau_0\epsilon^{-3},\tau_1\epsilon^{-3}],l^2(\mathbb{Z}))$ and satisfies
\begin{equation}
\label{bound-final-time}
\| w_{\epsilon}(t) - W(\epsilon (\cdot - t), \epsilon^3 t) \|_{l^2} +
\| p_{\epsilon}(t) - P_\epsilon(\epsilon (\cdot - t), \epsilon^3 t)  \|_{l^2}
\leq C_0 \epsilon^{3/2}, \quad t \in \left[-\tau_0 \epsilon^{-3},\tau_1 \epsilon^{-3}\right].
\end{equation}
\end{theo}

\begin{rem}
The Cauchy problem associated with the log--KdV equation (\ref{logKdV}) is not understood
in full generality: global solutions in some subspace of $H^1$
are constructed in \cite{CP14}, but the question of propagation of regularity remains open.
However, the classical approach (see for example Kato \cite{Kato}) allows to construct
short-time solutions with $H^s$ regularity, $s>3/2$, given initial data satisfying
a lower bound as in the assumptions of Theorem~\ref{theorem-justification}, namely $W\geq r_W >-1$
(in the neighborhood of which the nonlinearity $g$ is smooth).
\end{rem}

\begin{rem}
Using higher order asymptotic expansions and $\epsilon^K$-close initial data,
the approximation in \eqref{bound-final-time} could be improved to be $\mathcal{O}(\epsilon^K)$,
for any $K\in\N$ (see Remark~\ref{higherorder} below).
\end{rem}

\begin{rem}
Even if the travelling wave solution $W = W_{\rm stat}(\xi-\lambda\tau/2)$
to the log--KdV equation \eqref{LogKdV-background} is used
in bounds (\ref{bound-initial-time}) and (\ref{bound-final-time}),
where $W_{\rm stat}$ is a solution to the stationary log--KdV equation \eqref{stationaryLogKdV},
the results of Theorems \ref{theorem-stationary} and \ref{theorem-justification} do not
recover the result of Theorem \ref{theorem-stability}, because the small parameter $\delta$ in
Theorem \ref{theorem-stability} does not depend on the small parameter $\epsilon$.
\end{rem}

\section{Justification analysis for travelling waves}

Adopting the Fourier transform on $L^2(\mathbb{R})$ functions
\begin{eqnarray*}
\hat{w}(k) = \mathcal{F}(w)(k) := \int_{-\infty}^{\infty} w(z) e^{-ikz} dz
\end{eqnarray*}
with the inverse Fourier transform
\begin{eqnarray*}
w(z) = \mathcal{F}^{-1}(\hat{w})(z) := \frac{1}{2\pi} \int_{-\infty}^{\infty} \hat{w}(k) e^{ikz} dk,
\end{eqnarray*}
we can rewrite the existence problem (\ref{DFA}) as the fixed-point equation
\begin{equation}
\label{fixed-point}
w(z) = \frac{1}{1 + \mu} \int_{-1}^1 \Lambda(y) \tilde{V}_{\epsilon}'(w(z-y)) dy, \quad z \in \mathbb{R},
\end{equation}
where $\Lambda(z) = (1-|z|)_+$ is the hat function, or in the equivalent Fourier form
\begin{equation}
\label{fixed-point-Fourier}
\hat{w}(k) = \frac{1}{1 + \mu} \hat{\Lambda}(k) \mathcal{F}(\tilde{V}_{\epsilon}'(w))(k), \quad k \in \mathbb{R},
\end{equation}
where $\hat{\Lambda}(k) := \frac{4}{k^2} \sin^2\left(\frac{k}{2}\right)$.
This section presents the proof of Theorem \ref{theorem-stationary}, after several auxiliary results will be obtained.

\subsection{Nonzero solutions to the fixed-point equation (\ref{fixed-point})}

We shall first investigate if nonzero solutions to the fixed-point equation (\ref{fixed-point})
exist for $\mu = \mathcal{O}(\epsilon^2)$.
Therefore, we set $\mu := \epsilon^2 \lambda$ with an $\epsilon$-independent parameter $\lambda$.
The following proposition shows that, when $\lambda>1$ is fixed and $R>0$ is small enough,
there is no solution to the fixed-point equation \eqref{fixed-point} with norm in $L^2 \cap L^\infty$ less than $R$
other than the trivial (zero) solution.

\begin{prop}
\label{proposition-zero}
Set $\mu := \epsilon^2 \lambda$. For every $R>0$, there exists $\lambda_R > 1$
such that for all $\lambda > \lambda_R$ and all $\epsilon \in (0,1)$
the only solution to the fixed-point equation (\ref{fixed-point}) in
\begin{equation}
\label{ball-BR}
B_R := \{ w \in L^2(\mathbb{R}) \cap L^{\infty}(\mathbb{R}) : \;\; \| w \|_{L^2 \cap L^{\infty}}
\leq R, \quad w \geq 0 \}
\end{equation}
is the trivial zero solution. Furthermore, $\lambda_R$ may be chosen so that
$\lambda_R \Tend{R}{0} 1$.
\end{prop}

\begin{proof}
We write
$$
\tilde{V}_{\epsilon}'(w) = (1+w)^{1+\epsilon^2} - 1^{1+\epsilon^2} = (1+\epsilon^2) \int_0^w (1 + x)^{\epsilon^2} dx.
$$
Let $A_{\lambda,\epsilon}(w)$ denote the right-hand size of
the fixed-point equation (\ref{fixed-point}). Since $\| \Lambda \|_{L^1} = 1$ and
$\| \Lambda \|_{L^2} = \frac{\sqrt{2}}{\sqrt{3}} < 1$, we apply Young's inequality and obtain
\begin{eqnarray*}
\| A_{\lambda,\epsilon}(w) \|_{L^2 \cap L^{\infty}} & \leq & \frac{1}{1 + \epsilon^2 \lambda} \| \Lambda \|_{L^1 \cap L^2}
\|\tilde{V}_{\epsilon}'(w) \|_{L^2} \\
& \leq & \frac{1 + \epsilon^2}{1 + \epsilon^2 \lambda} (1 + \| w \|_{L^{\infty}})^{\epsilon^2} \| w \|_{L^2}
\end{eqnarray*}
Consider the ball given by (\ref{ball-BR}) of positive functions in $L^2(\mathbb{R}) \cap L^{\infty}(\mathbb{R})$ centered at zero
with the radius $R > 0$, denoted by $B_R$. If $R$ is fixed, there exists an $\epsilon$-independent constant $C_R$ such that
$$
(1 + \| w \|_{L^{\infty}})^{\epsilon^2} \leq 1 + C_R \epsilon^2 \log(1+R),
\quad \mbox{\rm for every } \epsilon \in (0,1).
$$
Furthermore, $C_R$ may be chosen so that $C_R \Tend{R}{0} 1$.

For $\lambda \geq \lambda_R := 1 + 2 C_R \log(1+R)$, we have $A_{\lambda,\epsilon} : B_R \to B_{R}$.
Moreover, using similar bounds
\begin{eqnarray*}
\| A_{\lambda,\epsilon}(w_1) - A_{\mu,\epsilon}(w_2) \|_{L^2} & \leq & \frac{1}{1 + \epsilon^2 \lambda} \| \Lambda \|_{L^1} \|\tilde{V}_{\epsilon}'(w_1) - \tilde{V}_{\epsilon}'(w_2)\|_{L^2} \\
& \leq & \frac{1 + \epsilon^2}{1 + \epsilon^2 \lambda} (1 + \max\{ \| w_1 \|_{L^{\infty}},\| w_2 \|_{L^{\infty}}\})^{\epsilon^2} \| w_1 - w_2\|_{L^2}\\
& \leq & \frac{1 + \epsilon^2}{1 + \epsilon^2 \lambda} (1 + C_R \epsilon^2 \log(1 + R)) \| w_1 - w_2\|_{L^2},
\end{eqnarray*}
we have the desired contraction property for the operator $A_{\lambda,\epsilon} : B_R \to B_R$
if $\lambda > \lambda_R$. Since $A_{\mu,\epsilon}(0) = 0$, the contraction principle
guarantees that the trivial solution $w = 0$ is the only fixed point of $A_{\lambda,\epsilon}$ in the set $B_R$.
\end{proof}

Next we set $\mu := \epsilon^2 \lambda$ with $\lambda \in (1,\infty)$ being fixed and $\epsilon$-independent.
Proposition \ref{proposition-zero} does not rule out the existence of nonzero solutions in $B_R$
to the fixed-point equation \eqref{fixed-point}
for sufficiently large $R$. In what follows, we will consider the
nonzero solutions to the fixed-point equation \eqref{fixed-point},
which are close to travelling waves given by the stationary log--KdV equation (\ref{stationaryLogKdV}).

Let us now recapture the formal limit to the stationary log--KdV equation (\ref{stationaryLogKdV}).
Using the Taylor series expansion as $k \to 0$,
\begin{equation}
\label{expansion-1}
\hat{\Lambda}(k) =  \frac{4}{k^2} \sin^2\left(\frac{k}{2}\right) = 1 - \frac{1}{12} k^2 + \mathcal{O}(k^4),
\end{equation}
and the power series (\ref{expansion-2}) for $\tilde{V}_{\epsilon}'(w)$,
we truncate the fixed-point equation (\ref{fixed-point-Fourier}) at the leading-order terms as follows
\begin{equation}
\label{stationaryLogKdVFourier}
\epsilon^2 \lambda \widehat{w_{\rm lead}}(k) =
- \frac{1}{12} k^2 \widehat{w_{\rm lead}}(k)
+ \epsilon^2 \mathcal{F}((1+w_{\rm lead}) \log(1+w_{\rm lead}))(k).
\end{equation}
Using the inverse Fourier transform and setting the variables $x = \epsilon z$
and $W(x) = w_{\rm lead}(z)$, we hence recover the stationary log--KdV equation
(\ref{stationaryLogKdV}).

\subsection{Solitary waves for the stationary log--KdV equation}

A standard construction of solitary waves for the stationary log--KdV equation
(\ref{stationaryLogKdV}) is based on a dynamical system analysis and gives the following result.

\begin{prop}
\label{proposition-soliton}
For any $\lambda > 1$, there exists a unique (up to the spatial translation) solution $W_{\rm stat}$
to the stationary log--KdV equation (\ref{stationaryLogKdV}) in $H^1(\mathbb{R})$,
such that $W_{\rm stat}(x) > 0$ for all $x \in \mathbb{R}$. Moreover, $W_{\rm stat} \in H^{\infty}(\mathbb{R})$,
$W_{\rm stat}'$ vanishes only at one point on $\mathbb{R}$, and
\begin{equation}
\label{bound-soliton}
W_{\rm stat}(x) \leq C_{\lambda} e^{- \kappa_{\lambda} |x|} ,\quad x \in \mathbb{R},
\end{equation}
for some $\lambda$-dependent positive constants $C_{\lambda}$ and $\kappa_{\lambda}$.
\end{prop}

\begin{proof}
Integrating the second-order differential equation (\ref{stationaryLogKdV}), we obtain
the energy
$$
E(W) := \frac{1}{24} \left(\frac{d W}{dx}\right)^2 + \frac{1}{2} (1+W)^2 \log(1+W) - \frac{1}{4} (1 + W)^2 - \frac{1}{2} \lambda W^2 = E_0,
$$
which is constant in $x$. Since any solution in $H^1(\mathbb{R})$ should decay to zero at infinity, we
set $E_0 = -\frac{1}{4}$. Because $E(W) \to \infty$ as $W \to \infty$, the turning point $W_0 > 0$ such that
$E(W_0) = E_0$ exists if $E(W)$ is concave near $W = 0$. This is ensured by the condition
$\lambda > 1$.

Further analysis of the nonlinear potential shows that if $\lambda > 1$, there is a unique turning point $W_0$
and a unique homoclinic orbit in the right-half of the phase plane $(W,W')$
that connects the saddle point $(0,0)$ for $E_0 = -\frac{1}{4}$. For this homoclinic orbit,
$W'$ vanishes at exactly one point $x_0$, where $W(x_0) = W_0$. By the ODE theory, the homoclinic orbit
for the nondegenerate saddle point decays exponentially fast at infinity with the precise decay rate
$\kappa_{\lambda} := \sqrt{12 (\lambda - 1)}$. Furthermore, bootstrapping arguments
for the differential equation (\ref{stationaryLogKdV}) yield $W_{\rm stat} \in H^{\infty}(\mathbb{R})$
because $W\mapsto\log(1 + W)$ is $C^{\infty}$ on $(0,\infty)$.
\end{proof}

\begin{rem}
By the translational symmetry, we can always shift $W_{\rm stat}$ so that $x_0 = 0$, in which case,
$W_{\rm stat}'(0) = 0$ and $W_{\rm stat}$ is even.
\end{rem}

Linearizing the nonlinear differential equation (\ref{stationaryLogKdV}) at the
solitary wave $W_{\rm stat}$, we obtain the Schr\"{o}dinger operator with a bounded and decaying potential
\begin{equation}
\label{Schrodinger-operator}
L_{\lambda} := - \frac{1}{12} \frac{\partial^2}{\partial x^2} + \lambda - 1 - \log(1+W_{\rm stat}) :
H^2(\mathbb{R}) \to L^2(\mathbb{R}).
\end{equation}
Although the exact location of the spectrum of $L_{\lambda}$ is unknown, several facts follow from
the Sturm theory (see Chapter 5.5 in \cite{Teschl} for review of the Sturm theory).

\begin{prop}
For any $\lambda > 1$, the spectrum of $L_{\lambda}$ in $L^2(\mathbb{R})$ includes one negative eigenvalue
$\lambda_{-1}$ with the positive eigenfunction $W_{-1}$ and the zero eigenvalue $\lambda_0 = 0$
with the eigenfunction $W_0 = W_{\rm stat}'$. The rest of the spectrum of $L_{\lambda}$ lies in $(0,\infty)$ and
is bounded away from zero by a positive number. Consequently, the linear operator $L_{\lambda}$ is invertible
with bounded inverse on the subspace of $L^2(\mathbb{R})$ which are $L^2$-orthogonal to $W_0$.
\label{proposition-operator}
\end{prop}

\begin{proof}
Since $L_\lambda$ is self-adjoint, it has a real spectrum.
The zero eigenvalue is due to the possible translation of the solitary wave $W_{\rm stat}$ in space:
$L_{\lambda} W_{\rm stat}'  = 0$. Since $W_{\rm stat}'$ has exactly one zero, there exists exactly one negative
eigenvalue $\lambda_{-1}$ with a positive eigenfunction $W_{-1}$: $L_{\lambda} W_{-1} = \lambda_{-1} W_{-1}$.
The continuous spectrum of $L_{\lambda}$ is bounded from below by the positive number $\lambda - 1$,
thanks to the fact that the potential $\log(1+W_{\rm stat})$ of the Schr\"{o}dinger operator $L_{\lambda}$
is bounded and exponentially decaying at infinity. By Sturm's theory, there may exist a finite number
of positive eigenvalues between $0$ and $\lambda - 1$.
\end{proof}

For iterations of the fixed-point equation (\ref{fixed-point}), it is more convenient to work
with the operator
\begin{equation}
\label{fixed-point-operator}
S_{\lambda} := \left( - \frac{1}{12} \frac{\partial^2}{\partial x^2} + \lambda - 1 \right)^{-1} \log(1+W_{\rm stat}) :
L^2(\mathbb{R}) \to H^2(\mathbb{R}).
\end{equation}
The following result is an equivalent reformulation of Proposition \ref{proposition-operator}.

\begin{prop}
For any $\lambda > 1$, the spectrum of $S_{\lambda}$ in $L^2(\mathbb{R})$
lies in $(0,\infty)$ and includes one simple eigenvalue $\mu_{-1}$ bigger than 1,
a simple eigenvalue $\mu_0 = 1$ with the eigenfunction $W_0 = W_{\rm stat}'$, and the rest of the spectrum
of $S_{\lambda}$ is located in the interval $(0,1)$ bounded away from $\mu_0 = 1$.
Consequently, the linear operator $I - S_{\lambda}$ is invertible with bounded inverse on the subspace
of functions in $L^2(\mathbb{R})$ orthogonal to $W_0$.
\label{proposition-fixed-point-operator}
\end{prop}

\begin{proof}
The operator $S_{\lambda}$ is conjugated via the positive operator
$\left( - \frac{1}{12} \partial_x^2 + \lambda - 1 \right)^{1/2}$
to a self-adjoint operator in $L^2(\mathbb{R})$. Hence the spectrum of
$S_{\lambda}$ is real. Moreover, since $\log(1+W_{\rm stat}(x)) > 0$ for all $x \in \mathbb{R}$,
the spectrum of $S_{\lambda}$ is positive.

By Sylvester's inertia law (see Chapter 4.1.2 in \cite{Pel-book}),
operators $L_{\lambda}$ and $I - S_{\lambda}$ have the same number of negative eigenvalues
and the same multiplicity of the zero eigenvalue. By Proposition \ref{proposition-operator},
$L_{\lambda}$ has one simple negative eigenvalue and a simple zero eigenvalue.
Equivalently, $S_{\lambda}$ has one simple eigenvalue
$\mu_{-1} > 1$ and a simple eigenvalue $\mu_0 = 1$.
Moreover, the eigenfunction of $S_{\lambda}$ for $\mu_0 = 1$ is the same
as that of $L_{\lambda}$ for $\lambda_0 = 0$.

Finally, because the spectrum of $L_{\lambda}$ on the orthogonal complement of
$X_0 := {\rm span}\{W_{-1},W_0\}$ in $L^2(\mathbb{R})$
is strictly positive and bounded away from zero, the rest of the spectrum of $S_{\lambda}$ is
located in the interval $(0,1)$ and bounded away from $\mu_0 = 1$. Consequently,
$\| S_{\lambda} \|_{X_0^{\perp} \to L^2} < 1$ and, by Neumann's theorem,
$I - S_{\lambda}$ is invertible with bounded inverse on the subspace $X_0^{\perp}$
which contains functions in $L^2(\mathbb{R})$ orthogonal to $X_0$. Furthermore, since $\mu_{-1} > 1$,
it is also invertible on the subspace of functions in $L^2(\mathbb{R})$ orthogonal to $W_0$.
\end{proof}

\begin{rem}
It follows from the criterion given by Pego \cite{pego} that $S_{\lambda}$ is actually
a compact operator in $L^2(\mathbb{R})$. However, we do not need to use this fact here
nor to construct the spectrum of $S_{\lambda}$ explicitly.
\end{rem}

\subsection{Strategy to prove Theorem \ref{theorem-stationary}}

Let us divide the infinite line for the Fourier variable $k$ into two sets $\mathcal{I}_p := [-\epsilon^{p},\epsilon^p]$
and $\mathcal{J}_p := \mathbb{R} \backslash \mathcal{I}_p$, where a positive $\epsilon$-independent
parameter $p$ is to be defined later. Let $\chi_S$ be the characteristic function of the set
$S \subset \mathbb{R}$. Then, we decompose the solution in the Fourier form into two parts:
\begin{equation}
\label{decomposition}
\hat{w}(k) = \hat{u}(k) + \hat{v}(k), \quad \mbox{\rm where} \;\;
\hat{u}(k) := \chi_{\mathcal{I}_p}(k) \hat{w}(k), \quad
\hat{v}(k) := \chi_{\mathcal{J}_p}(k) \hat{w}(k).
\end{equation}
The original problem (\ref{fixed-point-Fourier}) is now written as a system of two equations
\begin{equation}
\label{fixed-point-Fourier-1}
\hat{v}(k) = \frac{1}{1 + \epsilon^2 \lambda}  \chi_{\mathcal{J}_p}(k) \hat{\Lambda}(k) \mathcal{F}(\tilde{V}_{\epsilon}'(u+v))(k), \quad k \in \mathcal{J}_p
\end{equation}
and
\begin{equation}
\label{fixed-point-Fourier-2}
\hat{u}(k) = \frac{1}{1 + \epsilon^2 \lambda}  \chi_{\mathcal{I}_p}(k) \hat{\Lambda}(k) \mathcal{F}(\tilde{V}_{\epsilon}'(u+v))(k), \quad k \in \mathcal{I}_p.
\end{equation}
Here we set $\lambda > 1$ to be $\epsilon$-independent. For $R > 0$ and $r \in (-1,0)$, we define
\begin{equation}
\label{ball-R}
B_{R,r} := \{ u \in L^2(\mathbb{R}) \cap L^{\infty}(\mathbb{R}) : \;\; r \leq \inf_{\mathbb{R}} u, \quad \sup_{\mathbb{R}} u \leq R\},
\end{equation}
to consider functions which may have small negative and large positive values.

First, we show that for any $u \in B_{R,r}$ and for any small $\epsilon$, there exists a unique
solution $v$ to the first equation (\ref{fixed-point-Fourier-1}) such that
\begin{equation}
\label{bound-1}
\| v \|_{L^2 \cap L^{\infty}} \leq C_{R,r} \epsilon^{2-2p} \| u \|_{L^2},
\end{equation}
where the positive constant $C_{R,r}$ is independent of $\epsilon$ and $\| u \|_{L^2}$.
We use the contraction principle for equation (\ref{fixed-point-Fourier-1}) which holds if $p < 1$.

Second, we show that for any $v$ expressed from solution to the first equation
and for any small $\epsilon$, there exists a unique solution $u$ to the second equation
(\ref{fixed-point-Fourier-2}) near the solution $w_{\rm lead}=W_{\rm stat}(\epsilon\cdot)$
to the stationary log--KdV equation in the Fourier form (\ref{stationaryLogKdVFourier}):
\begin{equation}
\label{bound-2}
\| u - W_{\rm stat}(\epsilon \cdot) \|_{L^2 \cap L^{\infty}} \leq
C_{R,r,\lambda} \max\{\epsilon^{4p-2},\epsilon^{2-2p}\}  \| W_{\rm stat}(\epsilon \cdot) \|_{L^2},
\end{equation}
where the positive constant $C_{R,r,\lambda}$ is independent of $\epsilon$.
We use a fixed-point argument for equation (\ref{fixed-point-Fourier-2}).
Note that no contraction principle
can be directly applied directly neither to the full equation (\ref{fixed-point-Fourier})
nor to the reduced equation (\ref{fixed-point-Fourier-2})
because even if the fixed point exists, the nonlinear operator
on the right-hand side is not a contraction operator in the neighborhood
of the fixed point. This fact is explained roughly because the power of the nonlinear term
is bigger than one for any $\epsilon > 0$ and, in particular, it results in
the appearance of eigenvalue $\mu_{-1} > 1$ in Proposition \ref{proposition-fixed-point-operator}.
Therefore, we have to regroup the left-hand and right-hand side terms of equation
(\ref{fixed-point-Fourier-2}) before applying fixed-point arguments.

Note that $\| W_{\rm stat}(\epsilon \cdot) \|_{L^2} = \mathcal{O}(\epsilon^{-1/2})$ as $\epsilon \to 0$,
therefore, both corrections $u - W_{\rm stat}$ and $v$ are small in $L^{\infty}$ norm if
\begin{equation}
\label{constraints-on-p}
2 - 2p - \frac{1}{2} > 0 \quad \mbox{\rm and} \quad 4 p - 2 - \frac{1}{2} > 0,
\end{equation}
that is, for $\displaystyle p \in \left( \frac58,\frac68\right)$.
The optimal (smallest) bound occurs at $p = 2/3$ and corresponds to the power $1/6$
in the bound (\ref{bound-major}).
Thanks to the positivity of $W_{\rm stat}$, we have $r = \mathcal{O}(\epsilon^{1/6})$ as $\epsilon \to 0$.
At the same time, $R = \mathcal{O}(1)$ depends on $\lambda > 1$ and can be as large as necessary
(but $\epsilon$-independent).

We now follow the scheme above and prove bounds (\ref{bound-1}) and (\ref{bound-2}).
As explained above, these bounds yield Theorem \ref{theorem-stationary}.

\subsection{Proof of the bound (\ref{bound-1})}

The following lemma yields the bound (\ref{bound-1}).

\begin{lem}
For $R > 0$ and $r \in (-1,0)$, let $u$ belong to the set $B_{R,r}$ defined in \eqref{ball-R}.
For any $\lambda > 1$, $p \in (0,1)$, and sufficiently small $\epsilon$,
there exists a unique solution to equation (\ref{fixed-point-Fourier-1})
such that
\begin{equation}
\label{bound-1-again}
\| v \|_{L^2 \cap L^{\infty}} \leq C_{R,r} \epsilon^{2-2p} \| u \|_{L^2},
\end{equation}
where the positive constant $C_{R,r}$ is independent of $\epsilon$ and $\| u \|_{L^2}$.
Moreover, the map $u \mapsto v$ is $C^1$.
\label{lemma-bound-1}
\end{lem}

\begin{proof}
We write $\tilde{V}_{\epsilon}'(w) = w + N_{\epsilon}(w)$, where
$$
N_{\epsilon}(w) = (1+w)^{1+\epsilon^2} - 1 - w = \log(1 + w) \int_0^{\epsilon^2} (1 + w)^{1+x} dx
$$
and
$$
N_{\epsilon}(w_1) - N_{\epsilon}(w_2) = \epsilon^2 \int_{w_2}^{w_1} (1 + x)^{\epsilon^2} dx +
\int_{w_2}^{w_1} \log(1+x) \left( \int_0^{\epsilon^2} (1+x)^y dy \right) dx.
$$

The function $f(w):= \log(1+w)/w$ is strictly decreasing for $w > -1$ with $f(0) = 1$.
As a result, for every $r \in (-1,0)$, there is a positive constant $C_r$ such that
$$
|N_{\epsilon}(w)| \leq \epsilon^2 C_r (1 + w)^{1+\epsilon^2} w, \quad w \geq r
$$
and
$$
|N_{\epsilon}(w_1) - N_{\epsilon}(w_2)| \leq \epsilon^2 C_r \left(1 + \max\{w_1,w_2\}\right)^{1+\epsilon^2} |w_1-w_2|, \quad
w_1, w_2 \geq r.
$$
Note that $C_r$ may be chosen so that $C_r \Tend{r}{0} 1$.

Therefore, we rewrite equation (\ref{fixed-point-Fourier-1}) in the equivalent form
\begin{equation}
\label{outer-problem}
\hat{v}(k) = \hat{\mathcal{A}}_{\lambda,\epsilon}(\hat{u},\hat{v}) :=
\frac{1}{1 + \epsilon^2 \lambda}  \hat{\Lambda}_{\mathcal{J}_p}(k)
\left( \hat{v}(k) + \chi_{\mathcal{J}_p}(k) \mathcal{F}(N_{\epsilon}(u+v))(k) \right),
\quad k \in \mathcal{J}_p,
\end{equation}
where $\hat{\Lambda}_{\mathcal{J}_p}(k) :=  \chi_{\mathcal{J}_p}(k)  \hat{\Lambda}(k)$.
Because $|k| \geq \epsilon^p$ for $k \in \mathcal{J}_p$, we note from (\ref{expansion-1}) that
there exists an $\epsilon$-independent positive constant $C$ such that
$$
\| \hat{\Lambda}_{\mathcal{J}_p} \|_{L^{\infty}} \leq 1 - C \epsilon^{2p}.
$$
Let $\mathcal{A}_{\lambda,\epsilon}(u,v) := \mathcal{F}^{-1}(\hat{\mathcal{A}}_{\lambda,\epsilon}(\hat{u},\hat{v}))$.
By Plancherel's Theorem, we obtain
\begin{eqnarray*}
\| \mathcal{A}_{\lambda,\epsilon}(u,v) \|_{L^2} & = & \frac{1}{\sqrt{2\pi}} \| \hat{\mathcal{A}}_{\lambda,\epsilon}(\hat{u},\hat{v}) \|_{L^2} \\
& \leq & \| \hat{\Lambda}_{\mathcal{J}_p} \|_{L^{\infty}} \left(
\| v \|_{L^2} + \| N_{\epsilon}(u+v) \|_{L^2} \right) \\
& \leq & (1 - C \epsilon^{2p})
\left( \| v \|_{L^2} + \epsilon^2 C_r (1 + \| u+v \|_{L^{\infty}})^{1+\epsilon^2} \| u+v \|_{L^2} \right).
\end{eqnarray*}
By Cauchy--Schwarz inequality, we also have
\begin{eqnarray*}
\| \mathcal{A}_{\lambda,\epsilon}(u,v) \|_{L^{\infty}} & \leq & \frac{1}{2\pi}
\| \hat{\mathcal{A}}_{\lambda,\epsilon}(\hat{u},\hat{v}) \|_{L^1} \\
& \leq & \| \Lambda \|_{L^2} \left(
\| v \|_{L^2} + \| N_{\epsilon}(u+v) \|_{L^2} \right) \\
& \leq & \frac{\sqrt{2}}{\sqrt{3}}
\left( \| v \|_{L^2} + \epsilon^2 C_r (1 + \| u+v \|_{L^{\infty}})^{1+\epsilon^2} \| u+v \|_{L^2} \right).
\end{eqnarray*}

Let $u \in B_{R,r}$ defined by (\ref{ball-R}), where $R > 0$ and $r \in (-1,0)$ are fixed
independently from $\epsilon$. Recall that if $p < 1$, then $\epsilon^{2p} \gg \epsilon^2$ as $\epsilon \to 0$.
For every $u$ in $B_{R,r}$, $\lambda > 1$, and sufficiently small $\epsilon > 0$,
the operator $\mathcal{A}_{\lambda,\epsilon}(u,\cdot)$ maps a ball of functions
$v$ in $L^2(\mathbb{R}) \cap L^{\infty}(\mathbb{R})$ centered at zero
with the radius $\delta > 0$ to itself. Moreover, the operator
$\mathcal{A}_{\lambda,\epsilon}(u,\cdot)$ is a contraction in this ball,
using similar bounds
\begin{eqnarray*}
\| \mathcal{A}_{\lambda,\epsilon}(u,v_1) - \mathcal{A}_{\mu,\epsilon}(u,v_2) \|_{L^2} & \leq &
\| \hat{\Lambda}_{\mathcal{J}_p} \|_{L^{\infty}} \left(
\| v_1 - v_2 \|_{L^2} + \| N_{\epsilon}(u+v_1) - N_{\epsilon}(u+v_2) \|_{L^2} \right) \\
& \leq & (1 - C \epsilon^{2p})(1 + \epsilon^2 C_r (1 + R + \delta)^{1+\epsilon^2} )  \| v_1 - v_2\|_{L^2}
\end{eqnarray*}
and
\begin{eqnarray*}
\| \mathcal{A}_{\lambda,\epsilon}(u,v_1) - \mathcal{A}_{\mu,\epsilon}(u,v_2) \|_{L^{\infty}}
\leq  \frac{\sqrt{2}}{\sqrt{3}}
\left( 1 + \epsilon^2 C_r (1 + R + \delta)^{1+\epsilon^2} \right)  \| v_1 - v_2\|_{L^2}
\end{eqnarray*}
Again, the contraction in $L^2(\mathbb{R})$ is ensured by the fact that $\epsilon^{2p} \gg \epsilon^2$ as $\epsilon \to 0$.
Note that the Lipschitz constant is bounded from above by $1 - C \epsilon^{2p}$ independently
from $R$.

By the contraction mapping principle, for every given $u$ in $B_{R,r}$,
$\lambda > 1$, $p < 1$, and sufficiently small $\epsilon > 0$, there exists a unique fixed point of
the operator equation
$v = \mathcal{A}_{\lambda,\epsilon}(u,v)$ in $L^2(\mathbb{R}) \cap L^{\infty}(\mathbb{R})$ satisfying the bound
(\ref{bound-1-again}), where $\epsilon^{2p}$ is lost because of the proximity of the Lipschitz constant to unity.
Differentiability of the mapping $u \mapsto v$ also follows from the contraction mapping principle,
since the nonlinear operator $\mathcal{A}_{\lambda,\epsilon}(u,v)$ is differentiable for both $u$ and $v$.
\end{proof}

\subsection{Proof of the bound (\ref{bound-2})}

The following lemma yields the bound (\ref{bound-2}).

\begin{lem}
For any fixed $\lambda > 1$ and $p \in \left(\frac{5}{8},\frac{6}{8}\right)$, let $v \in L^2(\mathbb{R}) \cap L^{\infty}(\mathbb{R})$
be uniquely expressed in terms of $u \in B_{R,r}$ for some $R > 0$ and $r \in (-1,0)$
by Lemma \ref{lemma-bound-1}, where $B_{R,r}$ is defined by (\ref{ball-R}). For sufficiently small $\epsilon$,
there exists a unique solution to equation (\ref{fixed-point-Fourier-2}) in $B_{R,r}$ such that
\begin{equation}
\label{bound-2-again}
\| u - W_{\rm stat}(\epsilon \cdot) \|_{L^2 \cap L^{\infty}} \leq
C_{R,r,\lambda} \max\{\epsilon^{4p-2},\epsilon^{2-2p}\} \| W_{\rm stat}(\epsilon \cdot) \|_{L^2},
\end{equation}
where $W_{\rm stat}$ is the unique positive and even solution
to the stationary log--KdV equation \eqref{stationaryLogKdV} and
the positive constant $C_{R,r,\lambda}$ is independent of $\epsilon$.
\label{lemma-bound-2}
\end{lem}

\begin{proof}
By the Taylor expansion (\ref{expansion-1}), we can represent $\hat{\Lambda}(k)$ for any $k \in \mathcal{I}_p$ as
$$
\hat{\Lambda}(k) = \frac{1 + \hat{\Lambda}_{\rm Rem}(k)}{1+\frac{1}{12} k^2}, \quad |k| \leq \epsilon^{p},
$$
where the remainder term satisfies the bound
$$
\| \chi_{\mathcal{I}_p} \hat{\Lambda}_{\rm Rem} \|_{L^{\infty}} \leq C_{\Lambda} \epsilon^{4p},
$$
for a positive $\epsilon$-independent constant $C_{\Lambda}$. We now write
$\tilde{V}_{\epsilon}'(w) = w + \epsilon^2 (1+w) \log(1+w) + M_{\epsilon}(w)$, where
\begin{eqnarray*}
M_{\epsilon}(w) & = & (1+w)^{1+\epsilon^2} - 1 - w - \epsilon^2 (1+w) \log(1+w) \\
& = & \log^2(1 + w) \int_0^{\epsilon^2} \left( \int_0^x (1 + w)^{1+y} dy \right) dx.
\end{eqnarray*}
Recall that the function $f(w):= \log(1+w)/w$ is strictly decreasing for $w > -1$ with $f(0) = 1$.
Therefore, for any $r \in (-1,0)$, there is a positive constant $C_r$ such that
$$
|M_{\epsilon}(w)| \leq \frac{1}{2} \epsilon^4 C_r (1 + w)^{1+\epsilon^2} w^2, \quad w \geq r,
$$
and
$$
|M_{\epsilon}(w_1) - M_{\epsilon}(w_2)| \leq \epsilon^4 C_r \left(1 + \max\{w_1,w_2\}\right)^{1+\epsilon^2}
\max\{w_1,w_2\} |w_1-w_2|, \quad w_1, w_2 \geq r.
$$
Note again that $C_r$ may be chosen so that $C_r \Tend{r}{0} 1$.

Equation (\ref{fixed-point-Fourier-2}) can be rewritten in the equivalent form
\begin{equation}
\label{residual}
\left( \lambda + \frac{k^2}{12 \epsilon^2} \right) \hat{u}(k) -
\chi_{\mathcal{I}_p}(k) \mathcal{F}((1+u+v) \log(1+u+v))(k) =
\hat{H}_{\epsilon}(\hat{u},\hat{v})(k),
\end{equation}
where
$$
\hat{H}_{\epsilon}(\hat{u},\hat{v})(k) := -\frac{\lambda k^2}{12} \hat{u}(k)
+ \epsilon^{-2} \chi_{\mathcal{I}_p}(k) \mathcal{F}(M_{\epsilon}(u+v))(k)
+ \epsilon^{-2} \chi_{\mathcal{I}_p}(k) \hat{\Lambda}_{\rm Rem}(k) \mathcal{F}(\tilde{V}_{\epsilon}'(u+v))(k).
$$
It follows from the above estimates that for sufficiently small $\epsilon$,
the right-hand side of \eqref{residual} satisfies the estimate
\begin{eqnarray*}
\nonumber
\frac{1}{\sqrt{2\pi}} \| \hat{H}_{\epsilon}(\hat{u},\hat{v}) \|_{L^2} & \leq &
\frac{\lambda \epsilon^{2p}}{12} \| u \|_{L^2} + \frac{1}{2} \epsilon^2 C_r (1 + \| u + v \|_{L^{\infty}})^{1+\epsilon^2}
\| u+v \|_{L^{\infty}} \| u + v \|_{L^2} \\
\nonumber
& \phantom{t} & \phantom{text} + C_{\Lambda} \epsilon^{4p-2}
(1 + \epsilon^2 C_r (1 + \| u + v \|_{L^{\infty}})^{1+\epsilon^2})
\| u + v \|_{L^2}.
\end{eqnarray*}
Recall that if $p < 1$, then $\epsilon^{4p-2} \gg \epsilon^{2p} \gg \epsilon^2$ as $\epsilon \to 0$.
Let $v \in L^2(\mathbb{R}) \cap L^{\infty}(\mathbb{R})$
be uniquely expressed in terms of $u \in L^2(\mathbb{R}) \cap L^{\infty}(\mathbb{R})$
by Lemma \ref{lemma-bound-1}. Then, we obtain
\begin{eqnarray}
\frac{1}{\sqrt{2\pi}} \| \hat{H}_{\epsilon}(\hat{u},\hat{v}) \|_{L^2} \leq C_{R,r} \epsilon^{4p-2}
\left(1 + \epsilon^{2-2p} \| u \|_{L^2} \right)^{1+\epsilon^2} \| u \|_{L^2}, \label{tech-bound-1}
\end{eqnarray}
where the positive constant $C_{R,r}$ is independent of $\epsilon$ and $\| u \|_{L^2}$.
Since $\mathcal{I}_p$ is compact, we also have
\begin{eqnarray}
\label{tech-bound-2}
\frac{1}{2\pi} \| \hat{H}_{\epsilon}(\hat{u},\hat{v}) \|_{L^1} \leq
\frac{\epsilon^{p/2}}{\sqrt{2} \pi} \| \hat{H}_{\epsilon}(\hat{u},\hat{v}) \|_{L^2}.
\end{eqnarray}

Let us define the set 
$$
B_{R,r,C} := \left\{ u \in B_{R,r} : \quad \| u \|_{L^2} \leq C \epsilon^{-1/2} \right\},
$$
for some $\epsilon$-independent constant $C > \| W_{\rm stat} \|_{L^2}$.
If $p$ belongs to the bounds (\ref{constraints-on-p}) and $u$ belongs to $B_{R,r,C}$, 
then the term $\epsilon^{2-2p} \| u \|_{L^2}$ is bounded by a small constant as $\epsilon \to 0$.
For convenience, we will simply omit this term in the upper bounds. In what follows, 
we use a fixed-point argument in $B_{R,r,C}$, which ensures that $u$ satisfies (\ref{bound-2-again}).

Let $H_{\epsilon}(u,v) := \mathcal{F}^{-1} \hat{H}_{\epsilon}(\hat{u},\hat{v})$.
From (\ref{tech-bound-1}) and (\ref{tech-bound-2}) for $u \in B_{R,r,C}$, we have
\begin{equation}
\label{tech-bound-H}
\| H_{\epsilon}(u,v) \|_{L^2 \cap L^{\infty}} \leq C_{R,r} \epsilon^{4p-2} \| u \|_{L^2}.
\end{equation}
Since the mapping  $u \mapsto v$ is differentiable and all nonlinear functions
in $H_{\epsilon}(u,v)$ are differentiable both in $u$ and $v$, the remainder term $H_{\epsilon}(u,v)$ is differentiable
with respect to $u$ in $L^2(\mathbb{R}) \cap L^{\infty}(\mathbb{R})$.

Next, we study the left-hand side of (\ref{residual}). We write it as
$\hat{F}_{\epsilon}(\hat{u}) + \hat{G}_{\epsilon}(\hat{u},\hat{v})$, where
$$
\hat{F}_{\epsilon}(\hat{u})(k) := \left( \lambda + \frac{k^2}{12 \epsilon^2} \right) \hat{u}(k) -
\chi_{\mathcal{I}_p}(k) \mathcal{F}((1+u) \log(1+u))(k)
$$
and
$$
\hat{G}_{\epsilon}(\hat{u},\hat{v})(k) := -\chi_{\mathcal{I}_p}(k) \mathcal{F}((1+u+v) \log(1+u+v) - (1+u) \log(1+u))(k).
$$
Since the function $f(w) := (1+w)\log(1+w)$ is differentiable for any $w > -1$
with $f'(w) = 1 + \log(1+w)$, we have the bound
\begin{eqnarray*}
\frac{1}{\sqrt{2\pi}} \| \hat{G}_{\epsilon}(\hat{u},\hat{v}) \|_{L^2} & \leq & \| (1+u+v) \log(1+u+v) - (1+u) \log(1+u) \|_{L^2} \\
& \leq & (1 + C_r \| \log(1+u+v) \|_{L^{\infty}}) \| v \|_{L^2} \\
& \leq & (1 + C_r \|u+v \|_{L^{\infty}}) \| v \|_{L^2}.
\end{eqnarray*}
Using the bound (\ref{bound-1-again}) from Lemma \ref{lemma-bound-1}
and a similar bound for $\| \hat{G}_{\epsilon}(\hat{u},\hat{v}) \|_{L^1}$, 
we hence have for $u \in B_{R,r,C}$,
\begin{eqnarray}
\label{tech-bound-3}
\frac{1}{\sqrt{2\pi}} \| \hat{G}_{\epsilon}(\hat{u},\hat{v}) \|_{L^2 \cap L^1} \leq C_{R,r} \epsilon^{2-2p} \| u \|_{L^2}.
\end{eqnarray}
Let $G_{\epsilon}(u,v) := \mathcal{F}^{-1} \hat{G}_{\epsilon}(\hat{u},\hat{v})$.
From (\ref{tech-bound-3}), we have
\begin{equation}
\label{tech-bound-G}
\| G_{\epsilon}(u,v) \|_{L^2 \cap L^{\infty}} \leq C_{R,r} \epsilon^{2-2p} \| u \|_{L^2}.
\end{equation}
Again, $G_{\epsilon}(u,v(u))$ is differentiable with respect to $u$ in $L^2(\mathbb{R}) \cap L^{\infty}(\mathbb{R})$.

It remains to study the leading-order part $\hat{F}_{\epsilon}(\hat{u})$, where we apply arguments based on
the implicit function theorem. Let us define $F_{\epsilon}(u) := \mathcal{F}^{-1}(\hat{F}_{\epsilon}(\hat{u}))$. For any $\epsilon > 0$,
the nonlinear operator $F_{\epsilon}(u)$ is a bounded operator from a subset of
$L^2(\mathbb{R}) \cap L^{\infty}(\mathbb{R})$ to $L^2(\mathbb{R}) \cap L^{\infty}(\mathbb{R})$ thanks to the bounds
\begin{eqnarray*}
\frac{1}{\sqrt{2\pi}} \| \hat{F}_{\epsilon}(\hat{u}) \|_{L^2} \leq
\left( \lambda + \frac{1}{12 \epsilon^{2(1-p)}} + C_r (1  + \| u \|_{L^{\infty}}) \right) \| u \|_{L^2}
\end{eqnarray*}
and a similar bound for $\| \hat{F}_{\epsilon}(\hat{u}) \|_{L^1}$. 
The map $u \mapsto F_{\epsilon}(u)$ is $C^{\infty}$ thanks to the smoothness of the function
$u\mapsto\log(1+u)$ on $(-1,\infty)$.

Consider the solitary wave $W_{\rm stat}$ solution to the stationary log--KdV equation (\ref{stationaryLogKdV})
from Proposition \ref{proposition-soliton}, and let $w_{\rm lead} = W_{\rm stat}(\epsilon \cdot)$
be the corresponding solution to \eqref{stationaryLogKdVFourier}. We have the relationship between the
Fourier transforms of $w_{\rm lead}$ and $W_{\rm stat}$:
$$
\widehat{w}_{\rm lead}(k) = \int_{\infty}^{\infty} W_{\rm stat}(\epsilon z) \, e^{-i k z} dz =
\frac{1}{\epsilon} \widehat{W_{\rm stat}}\left(\frac{k}\epsilon\right).
$$
We further define an approximation of $W_{\rm stat}$ by truncating the Fourier transform
$\widehat{w_{\rm lead}}$ on $\mathcal{I}_p$, that is,
\begin{equation}
\label{approximation-Fourier}
W_{\rm app}(z) := \mathcal{F}^{-1}(\chi_{\mathcal{I}_p} \hat{w}_{\rm sol})(z) =
\frac{1}{2\pi} \int_{\mathcal{I}_p} \widehat{w_{\rm lead}}(k) e^{ik z} dk
= \frac{1}{2\pi}
\int_{-\epsilon^{p-1}}^{\epsilon^{p-1}} \widehat{W_{\rm stat}}(\kappa) e^{i \kappa \epsilon z} d \kappa.
\end{equation}
Since $W_{\rm stat} \in H^{\infty}(\mathbb{R})$ by Proposition \ref{proposition-soliton},
Sobolev's embedding implies that $W_{\rm stat} \in C^{\infty}(\mathbb{R})$,
which then implies that $\widehat{W_{\rm stat}}$ decays faster than any power at infinity.
It follows from (\ref{approximation-Fourier}) for $p < 1$ that
the integration interval extends to the entire line as $\epsilon \to 0$.
As a result, for any $s > 0$, we have an $\epsilon$-independent positive constant $C_s$
such that for all sufficiently small $\epsilon > 0$,
\begin{equation}
\label{tech-bound-4}
\| W_{\rm app} - W_{\rm stat} \|_{L^2 \cap L^{\infty}} \leq C_s \epsilon^s.
\end{equation}
The nonlinear operator $F_{\epsilon}(u)$ evaluated at $u = W_{\rm app}$ is given
in the Fourier form by
$$
\mathcal{F} [F_{\epsilon}(W_{\rm app})](k) =
\chi_{\mathcal{I}_p}(k) \mathcal{F}((1+W_{\rm app}) \log(1+W_{\rm app}))(k) -
\chi_{\mathcal{I}_p}(k) \mathcal{F}((1+W_{\rm stat}) \log(1+W_{\rm stat}))(k).
$$
Consequently, thanks to the smoothness of the map $u \mapsto F_{\epsilon}(u)$
in $L^2(\mathbb{R}) \cap L^{\infty}(\mathbb{R})$ and the bound (\ref{tech-bound-4}),
we obtain
\begin{equation}
\label{tech-bound-5}
\| F_{\epsilon}(W_{\rm app}) \|_{L^2 \cap L^{\infty}} \leq C_s \epsilon^s,
\end{equation}
for any $s > 0$ and sufficiently small $\epsilon$.

We rewrite equation (\ref{residual}) as the implicit equation
\begin{equation}
\label{residual-new}
f_{\epsilon}(u) = h_{\epsilon}(u,v),
\end{equation}
where
\begin{eqnarray*}
f_{\epsilon}(u)
& := & \mathcal{F}^{-1} \chi_{\mathcal{I}_p} \left(\lambda - 1 + \frac{k^2}{12 \epsilon^2} \right)^{-1}
\left( \hat{F}_{\epsilon}(\hat{u}) - \hat{F}_{\epsilon}(\widehat{W_{\rm app}}) \right), \\
h_{\epsilon}(u,v)
& := & \mathcal{F}^{-1} \chi_{\mathcal{I}_p} \left(\lambda - 1 + \frac{k^2}{12 \epsilon^2} \right)^{-1}
\left( \hat{H}_{\epsilon}(\hat{u},\hat{v}) - \hat{G}_{\epsilon}(\hat{u},\hat{v})
- \hat{F}_{\epsilon}(\widehat{W_{\rm app}}) \right).
\end{eqnarray*}
Since $\lambda > 1$, we infer from the bounds (\ref{tech-bound-H}), (\ref{tech-bound-G})
and (\ref{tech-bound-5}) that for $u \in B_{R,r,C}$, 
\begin{equation}
\label{tech-bound-11}
\| h_{\epsilon}(u,v) \|_{L^2 \cap L^{\infty}} \leq C_{R,r,\lambda} \max\{\epsilon^{4p-2},\epsilon^{2-2p}\}  \| u \|_{L^2},
\end{equation}
where the positive constant $C_{R,r,\lambda}$ is independent of $\epsilon$ and $\| u \|_{L^2}$.
Therefore, the right-hand side of (\ref{residual-new}) is small in $L^2(\mathbb{R}) \cap L^{\infty}(\mathbb{R})$ norm,
if $p$ satisfies the bounds (\ref{constraints-on-p}) and $u$ belongs to $B_{R,r,C}$. 
The left-hand side of (\ref{residual-new}) is zero at $u = W_{\rm app}$.

Let us now consider the linearization operator $\partial_u f_{\epsilon}(W_{\rm app})$. In
the Fourier form, the linearization operator acting on $U$ is given by
$$
\mathcal{F} [\partial_u f_{\epsilon}(W_{\rm app}) U](k) :=
\hat{U}(k) - \chi_{\mathcal{I}_p}(k)
\left( \lambda - 1 + \frac{k^2}{12 \epsilon^2} \right)^{-1}
 \mathcal{F}(\log(1+W_{\rm app}) U)(k) .
$$
Note that $W_{\rm app}(z)$ is an even function of $z$ if $W_{\rm stat}(x)$ is an even
function of $x$ because $\widehat{W_{\rm stat}}(k)$ is an even function of $k$ and
the truncation in the Fourier domain is taken symmetrically around $k = 0$.
Also note that the fixed-point problem (\ref{fixed-point-Fourier}) preserves
the parity property in the space of even functions. Therefore, we can consider $u$ or $U := u - W_{\rm stat}$
in the space of even functions.

Recall the unbounded Schr\"{o}dinger operator $L_{\lambda}$ given by (\ref{Schrodinger-operator})
and the bounded operator $S_{\lambda}$ given by (\ref{fixed-point-operator}).
Let us now define the bounded operator $S_{\lambda,p}$ in the Fourier form by
$$
[\hat{S}_{\lambda,p} \hat{U}](k) :=
\chi_{\mathcal{I}_p}(k) \left( \lambda - 1 + \frac{k^2}{12 \epsilon^2} \right)^{-1}
\mathcal{F}(\log(1+W_{\rm stat}) U)(k).
$$
We obtain the bound
\begin{eqnarray*}
\frac{1}{2\pi} \| (\hat{S}_{\lambda} - \hat{S}_{\lambda,p}) \hat{U} \|^2_{L^2} & = &
\frac{1}{2\pi} \int_{\mathcal{J}_p} \frac{1}{(\lambda - 1 + \frac{k^2}{12 \epsilon^2})^2}
\left| \mathcal{F}(\log(1+W_{\rm stat}) U)(k) \right|^2 dk \\
& \leq & (12 \epsilon^{2-2p})^2 \| \log(1+W_{\rm stat}) U \|^2_{L^2},
\end{eqnarray*}
which yields, thanks to the positivity of $W$,
\begin{eqnarray}
\label{tech-bound-9}
\| (S_{\lambda} - S_{\lambda,p}) U \|_{L^2} \leq
12 \epsilon^{2- 2p} \| W_{\rm stat} \|_{L^{\infty}} \| U \|_{L^2}.
\end{eqnarray}

By Proposition \ref{proposition-fixed-point-operator}, the linear operator $I - S_{\lambda}$
is invertible with bounded inverse on the subspace of even functions in $L^2(\mathbb{R})$.
Thanks to the bound (\ref{tech-bound-9}), the linear operator $I - S_{\lambda,p}$
is also invertible with bounded inverse on the subspace of even functions in $L^2(\mathbb{R})$.
Finally, thanks to the bound (\ref{tech-bound-4}), the linearized operator  $\partial_u f_{\epsilon}(W_{\rm app})$ is
also invertible with bounded inverse on the subspace of even functions in $L^2(\mathbb{R})$.
In other words, there is a positive $\epsilon$-independent constant $C_{\lambda}$
such that for any sufficiently small
$\epsilon$ and any even function $h$ in $L^2(\mathbb{R})$, we have
$$
\| \left[ \partial_u f_{\epsilon}(W_{\rm app}) \right]^{-1} h \|_{L^2} \leq C_{\lambda}  \| h \|_{L^2}.
$$
Since $\mathcal{I}_p$ is compact,  we then have
\begin{equation}
\label{tech-bound-6}
\|\left[ \partial_u f_{\epsilon}(W_{\rm app}) \right]^{-1} h \|_{L^2 \cap L^{\infty}} \leq C_{\lambda}  \| h \|_{L^2}.
\end{equation}

Writing $u = W_{\rm app} + U$, we can now apply the standard implicit function theorem
to obtain a unique solution $U$ to the implicit equation
(\ref{residual-new}) in $L^2(\mathbb{R}) \cap L^{\infty}(\mathbb{R})$
close to the zero solution for small $\epsilon > 0$. In view of the
bounds (\ref{tech-bound-11}) and (\ref{tech-bound-6}), the solution satisfies the bound
$$
\| U \|_{L^2 \cap L^{\infty}} \leq
C_{R,r,\lambda} \max\{\epsilon^{4p-2},\epsilon^{2-2p}\} \| W_{\rm app}(\epsilon \cdot) \|_{L^2},
$$
where the positive constant $C_{R,r,\lambda}$ is $\epsilon$-independent. This bound yields (\ref{bound-2-again})
thanks to the proximity between $W_{\rm stat}$ and $W_{\rm app}$ given by the bound (\ref{tech-bound-4}).
\end{proof}


\subsection{Proof of the bound (\ref{bound-major-derivatives})}

Here we prove the bound (\ref{bound-major-derivatives}) for $k = 1$.
The proof extends to every $k \in \mathbb{N}$ by similar arguments and by induction.

To control the derivative of the solution $w = u + v$ of the fixed-point equation (\ref{fixed-point})
in the $L^2(\mathbb{R}) \cap L^{\infty}(\mathbb{R})$ norm,
we multiply the two equations (\ref{fixed-point-Fourier-1}) and (\ref{fixed-point-Fourier-2})
by $k$. After multiplication by $k$, equation (\ref{outer-problem}) for $k \in \mathcal{J}_p$ is rewritten in the form
\begin{equation}
\label{outer-problem-derivative}
k \hat{v}(k) = \frac{1}{1 + \epsilon^2 \lambda}  \hat{\Lambda}_{\mathcal{J}_p}(k)
\left( k \hat{v}(k) + \chi_{\mathcal{J}_p}(k) k \mathcal{F}(N_{\epsilon}(u+v))(k) \right),
\quad k \in \mathcal{J}_p.
\end{equation}
We repeat the estimates in the proof of Lemma \ref{lemma-bound-1},
after commutation of the Fourier transform and of the multiplication operator by $k$,
which becomes the derivative operator w.r.t. $z$ applied to
the nonlinear function $N_{\epsilon}(w)$. The nonlinearity $w \mapsto N_{\epsilon}(w)$ is smooth
since $w \geq r > -1$.

Let $B_{R,r}'$ be the set
\begin{equation}
\label{ball-R-derivative}
B_{R,r}' := \{ u \in B_{R,r} : \,\,
\partial_z u \in L^2(\mathbb{R}) \cap L^{\infty}(\mathbb{R}), \,\,
\| \partial_z u \|_{L^{\infty}} \leq R \}.
\end{equation}
By the same technique as in the proof of Lemma \ref{lemma-bound-1}, we obtain from equation
(\ref{outer-problem-derivative}) for sufficiently small $\epsilon$ that
for any fixed $\lambda > 1$, $p \in (0,1)$, $R > 0$, $r \in (-1,0)$, and any $u \in B_{R,r}'$,
the unique solution to equation (\ref{outer-problem}) satisfies, in
addition to the bound (\ref{bound-1-again}),
\begin{equation}
\label{bound-1-derivative}
\| \partial_z v \|_{L^2 \cap L^{\infty}} \leq C_{R,r} \epsilon^{2-2p} \| \partial_z u \|_{L^2},
\end{equation}
where the positive constant $C_{R,r}$ is independent of $\epsilon$ and $\| u \|_{H^1}$.

We then proceed with analysis of equation (\ref{residual}), which we also multiply by $k$.
From the same arguments as in the proof of Lemma \ref{lemma-bound-2}, we obtain, in addition
to the bounds (\ref{tech-bound-H}) and (\ref{tech-bound-G}),
\begin{equation}
\label{tech-bound-H-derivative}
\| \partial_z H_{\epsilon}(u,v) \|_{L^2 \cap L^{\infty}} \leq C_{R,r} \epsilon^{4p-2} \| \partial_z u \|_{L^2}
\end{equation}
and
\begin{equation}
\label{tech-bound-G-derivative}
\| \partial_z G_{\epsilon}(u,v) \|_{L^2 \cap L^{\infty}} \leq C_{R,r} \epsilon^{2-2p} \| \partial_z u \|_{L^2},
\end{equation}
where $v$ is expressed from equation (\ref{outer-problem})
using the bounds (\ref{bound-1-again}) and (\ref{bound-1-derivative}). Applying now
derivative in $z$ to the implicit equation (\ref{residual-new}), we obtain
\begin{equation}
\label{residual-derivative}
\partial_z f_{\epsilon}(u) = \partial_z h_{\epsilon}(u,v),
\end{equation}
where $\partial_z h_{\epsilon}(u,v)$ satisfies the bound
\begin{equation}
\label{tech-bound-derivative-1}
\| \partial_z h_{\epsilon}(u,v) \|_{L^2 \cap L^{\infty}} \leq
C_{R,r,\lambda} \max\{\epsilon^{4p-2},\epsilon^{2-2p}\} \| \partial_z u \|_{L^2}.
\end{equation}
The derivative of the linearized operator
$\partial_u f_{\epsilon}(W_{\rm app})$ applied to $U$ is, by the product rule,
$$
\partial_z \left( \partial_u f_{\epsilon}(W_{\rm app}) \, U \right) =
\partial_u f_{\epsilon}(W_{\rm app}) \partial_z U
+ \left( \partial_z \partial_u f_{\epsilon}(W_{\rm app}) \right) U,
$$
where the second term is bounded as
\begin{equation}
\label{tech-bound-derivative-2}
\|  \left( \partial_z \partial_u f_{\epsilon}(W_{\rm app}) \right) U \|_{L^2}
\leq C_{R,r,\lambda} \| \partial_z W_{\rm app}(\epsilon \cdot) \|_{L^{\infty}} \| U \|_{L^2}.
\end{equation}
Using the bounds (\ref{tech-bound-6}), (\ref{tech-bound-derivative-1})
and (\ref{tech-bound-derivative-2}), we obtain
from equation (\ref{residual-derivative}) for sufficiently small $\epsilon$ that
for any fixed $\lambda > 1$, the unique solution to equation (\ref{residual-new})
satisfies, in addition to the bound (\ref{bound-2-again}),
\begin{equation*}
\| \partial_z u - \partial_z W_{\rm stat}(\epsilon \cdot) \|_{L^2 \cap L^{\infty}} \leq
C_{R,r,\lambda} \max\{\epsilon^{4p-2},\epsilon^{2-2p}\} \| \partial_z W_{\rm app}(\epsilon \cdot) \|_{L^2}
\end{equation*}
where the positive constant $C_{R,r,\lambda}$ is independent of $\epsilon$. This bound yields
(\ref{bound-major-derivatives}) for $k = 1$, since the error bound is optimal for $p = \frac{2}{3}$ and
$$
\| \partial_z W_{\rm app}(\epsilon \cdot) \|_{L^2} \leq C \epsilon^{1-1/2},
$$
where the positive constant $C$ is independent of $\epsilon$.

\section{Stability of FPU travelling waves near the log--KdV limit}

This section presents the proof of Theorem \ref{theorem-stability}.

Let $(w_{\rm trav},p_{\rm trav}) \in C^1(\mathbb{R},l^2(\mathbb{Z}))$
denote the travelling wave (\ref{solitary-wave-FPU}) solution to the FPU lattice (\ref{FPU-lattice})
with the squared speed $c^2 = 1 + \epsilon^2 \lambda$. The amplitudes $(w_{\rm stat},p_{\rm stat})$
of the travelling wave are solutions to the system of advance equations
\begin{equation}
\label{diff-eq-q}
\left\{ \begin{array}{l}
-c w_{\rm stat}'(z) = p_{\rm stat}(z+1) - p_{\rm stat}(z), \\
- c p_{\rm stat}'(z) = \tilde{V}'_{\epsilon}(w_{\rm stat}(n-ct))
- \tilde{V}'_{\epsilon}(w_{\rm stat}(n-1-ct)),
\end{array} \right.
\quad z \in \mathbb{R}.
\end{equation}
Properties of $w_{\rm stat}$ are described by Theorem \ref{theorem-stationary}
for sufficiently small $\epsilon$.

For any fixed $c$, we decompose
$$
w(t) = w_{\rm trav}(t) + \mathcal{W}(t), \quad p(t) = p_{\rm trav}(t) + \mathcal{P}(t),
$$
and rewrite the system of FPU lattice equations (\ref{FPU-lattice}) in the perturbed form
\begin{eqnarray}
\left\{
\begin{split}
& \dot{\mathcal{W}}_n = \mathcal{P}_{n+1}-\mathcal{P}_n, \\
& \dot{\mathcal{P}}_n = \tilde{V}_{\epsilon}''(w_{\rm stat}(n-ct)) \mathcal{W}_{n}
- \tilde{V}_{\epsilon}''(w_{\rm stat}(n-1-ct))\mathcal{W}_{n-1} \\
& \phantom{text}  + \frac{1}{2} \tilde{V}_{\epsilon}'''(w_{\rm stat}(n-ct)) \mathcal{W}^2_{n}
- \frac{1}{2} \tilde{V}_{\epsilon}'''(w_{\rm stat}(n-1-ct))\mathcal{W}^2_{n-1} \\
& \phantom{text} + R_n(\mathcal{W}), \end{split} \right.
\label{FPU-lattice-perturbed}
\end{eqnarray}
where the remainder term is cubic in $\mathcal{W}$ thanks to the smoothness of $\tilde{V}_{\epsilon}$
on $(-1,\infty)$.
Therefore, in the perturbed form (\ref{FPU-lattice-perturbed}), it is assumed that the
solution $w$ remains within the a priori bounds (\ref{apriori-bound-time}),
which happens if $\mathcal{W}_n$ is sufficiently small for every $n \in \mathbb{Z}$.

Let $B_{\rho}$ denote a small ball in $l^2(\mathbb{Z})$ centered at zero with radius $\rho > 0$.
Thanks to the embedding of
$l^2(\mathbb{Z})$ into $l^{\infty}(\mathbb{Z})$, for any $\rho > 0$, there is a positive
constant $C_{\rho}$ such that the remainder term satisfies the bound
\begin{equation}
\label{bound-stability-1}
\| R(\mathcal{W}) \|_{l^2}
\leq C_{\rho} \;\sup_{z \in \mathbb{R}} |\tilde{V}_{\epsilon}''''(w(z))| \;\| \mathcal{W} \|_{l^2}^3.
\end{equation}
In what follows, $C_\rho$ denotes a positive constant that depends only on $\rho$
and remains bounded as $\rho \to 0$.
Similarly to (\ref{FPU-lattice-perturbed}), we expand the energy (\ref{FPU-energy}) near the travelling wave
\begin{eqnarray}
\label{FPU-energy-perturbed}
H = H_0 + H_1 + H_2 + H_R,
\end{eqnarray}
where
\begin{eqnarray*}
H_0 & = & \frac{1}{2} \sum_{n \in \mathbb{Z}} p_{\rm stat}^2(n - ct) +
\sum_{n \in \mathbb{Z}} \tilde{V}_{\epsilon}(w_{\rm stat}(n-ct)), \\
H_1 & = & \sum_{n \in \mathbb{Z}} p_{\rm stat}(n - ct) \mathcal{P}_n +
\sum_{n \in \mathbb{Z}} \tilde{V}'_{\epsilon}(w_{\rm stat}(n-ct)) \mathcal{W}_n, \\
H_2 & = & \frac{1}{2} \sum_{n \in \mathbb{Z}} \mathcal{P}_n^2 +
\frac{1}{2} \sum_{n \in \mathbb{Z}} \tilde{V}''_{\epsilon}(w_{\rm stat}(n-ct)) \mathcal{W}_n^2,
\end{eqnarray*}
and the remainder term $H_R$ satisfies the bound
\begin{equation}
\label{bound-stability-2}
|H_R| \leq C_{\rho} \;\sup_{z \in \mathbb{R}} |\tilde{V}_{\epsilon}'''(w_{\rm stat}(z))| \;\| \mathcal{W} \|_{l^2}^3.
\end{equation}

From the time conservation of $H$, it follows that $H_0$ is independent of $t$. This can be checked by explicit
differentiation, using the system (\ref{diff-eq-q}),
\begin{eqnarray*}
\frac{d H_0}{dt} & = & \sum_{n \in \mathbb{Z}} p_{\rm stat}(n-ct)
\left[ - c p_{\rm stat}'(n - ct) + \tilde{V}'_{\epsilon}(w_{\rm stat}(n-1-ct)) - \tilde{V}'_{\epsilon}(w_{\rm stat}(n-ct)) \right] = 0,
\end{eqnarray*}
On the other hand, $H_1$ is no longer constant. Using (\ref{diff-eq-q}) and (\ref{FPU-lattice-perturbed}),
we obtain
\begin{eqnarray}
\nonumber
\frac{d H_1}{d t}
& = & \sum_{n \in \mathbb{Z}} \left[ p_{\rm stat}(n-ct) \dot{\mathcal{P}}_n
+ \tilde{V}_{\eps}''(w_{\rm stat}(n-ct)) (p_{\rm stat}(n+1-ct) - p_{\rm stat}(n-ct)) \mathcal{W}_n \right] \\
& = & \frac{c}{2} \sum_{n \in \mathbb{Z}} w_{\rm stat}'(n - ct) \tilde{V}_{\epsilon}'''(w_{\rm stat}(n - ct)) \mathcal{W}_n^2
+ S_R, \label{balance-H1}
\end{eqnarray}
where the remainder term satisfies the bound
\begin{equation}
\label{bound-stability-3}
|S_R| \leq C_{\rho} \; \sup_{z \in \mathbb{R}} | \tilde{V}_{\epsilon}''''(w_{\rm stat}(z)) w_{\rm stat}'(z) | \;
\| \mathcal{W} \|_{l^2}^3,
\end{equation}
which follows from the bound (\ref{bound-stability-1}).

We shall now recall that
$$
\tilde{V}_{\epsilon}''(w) = (1+\epsilon^2) (1 + w)^{\epsilon^2}, \quad
\tilde{V}_{\epsilon}'''(w) = \epsilon^2 (1+\epsilon^2) (1 + w)^{\epsilon^2-1},
$$
and so on. As a result, $H_2$ is a convex quadratic form with the lower bound
\begin{equation}
\label{bound-stability-4}
H_2 \geq \frac{1}{2} \| \mathcal{P} \|_{l^2}^2 + \frac{1}{2} \| \mathcal{W} \|_{l^2}^2.
\end{equation}

Using the bound (\ref{bound-major-derivatives}) for $k = 1$, bounds (\ref{bound-stability-3}) and (\ref{bound-stability-4}),
we can estimate the balance equation (\ref{balance-H1}) as follows:
\begin{eqnarray*}
\left| \frac{d H_1}{d t} \right| & \leq  & C_{\rho} \epsilon^3 (1 + \rho) \| \mathcal{W} \|_{l^2}^2
\leq 2 C_{\rho} \epsilon^3 (1 + \rho) H_2,
\end{eqnarray*}
as long as $\| W \|_{l^2} \leq \rho$, where the positive constant $C_{\rho}$ is independent of $\epsilon$.
As a result, we obtain the lower bound
\begin{eqnarray}
\label{bound-stability-5}
H_1(t) - H_1(0) \geq  - 2 C_{\rho} \epsilon^3 (1 + \rho) \int_0^{|t|} H_2(t') dt'.
\end{eqnarray}
Now, using the energy expansion (\ref{FPU-energy-perturbed}) as well as the bounds (\ref{bound-stability-2})
and (\ref{bound-stability-5}), we can write
\begin{equation}
H - H_0 - H_1(0) \geq
- 2 C_{\rho} \epsilon^3 (1 + \rho) \int_0^{|t|} H_2(t') dt'
+ H_2(t) (1 - C_{\rho} \epsilon^2 \rho).
\end{equation}
By Gronwall's inequality, we obtain
\begin{equation}
\label{bound-stability-6}
H_2(t) \leq \frac{H - H_0 - H_1(0)}{1 - C_{\rho} \epsilon^2 \rho} e^{\tilde{C}_{\rho} \epsilon^3 |t|},
\end{equation}
where $\tilde{C}_{\rho}$ is another positive $\epsilon$-independent constant. Since $H - H_0 - H_1(0)$ is
$t$-independent, we can express it at $t = 0$ by
\begin{equation}
\label{bound-stability-7}
H - H_0 - H_1(0) = H_2(0) + H_R(0) \leq \tilde{\tilde{C}}^2_{\rho} \delta^2,
\end{equation}
where $\tilde{\tilde{C}}^2_{\rho}$ is yet another positive $\epsilon$-independent constant 
and the initial bound (\ref{bound-initial}) is used.
When $\tau_0>0$ is given, the bounds (\ref{bound-stability-4}),
(\ref{bound-stability-6}), and (\ref{bound-stability-7})
imply the stability bound (\ref{bound-final}) for $\epsilon\in(0,\epsilon_0)$ and  $\tau\in [-\tau_0,\tau_0]$,
with sufficiently small constants $\epsilon_0>0$ and $\delta_0\in(0,1)$ 
with 
$$
C_0 > 2 \tilde{\tilde{C}}_{\rho} e^{\frac{1}{2} \tilde{C}_{\rho} \tau_0}.
$$
Theorem \ref{theorem-stability} is proved in the ball $B_{\rho} \subset l^2(\mathbb{Z})$ 
with the radius $\rho := C_0 \delta$.

\section{Justification analysis for time-dependent solutions}

This section presents the proof of Theorem \ref{theorem-justification}.
In fact, it is a modification of the arguments in the proof of Theorem \ref{theorem-stability}.
The arguments follow quite closely to the method described by Schneider and Wayne \cite{SW00}, where interactions of
counter-propagating waves have also been included. We add this section for completeness, as well as
for comparison with stability theory of travelling waves in FPU lattices as described by KdV-type equations.

From the assumptions of Theorem \ref{theorem-justification}, we know there exist constants
$r_W$ and $R_W$ such that
\begin{equation} \label{boundsW}
-1 < r_W \leq W(\xi,\tau) \leq R_W,
\quad \xi \in \mathbb{R}, \; \tau \in [-\tau_0,\tau_1].
\end{equation}
For $\epsilon_0>0$ small enough, for all $\epsilon\in(0,\epsilon_0)$,
initial data $(w_{{\rm ini},\epsilon},p_{{\rm ini},\epsilon})$ satisfying the bound \eqref{bound-initial-time}
are such that all the terms in the sequence $w_{{\rm ini},\epsilon}$ are greater than some $r>-1$ independent of
$\epsilon$. Thus there exists a solution $(w,p)\in C^1([-T_0,T_1],l^2(\mathbb{Z}))$ to
the FPU lattice equations \eqref{FPU-lattice},
at least for small times $T_0,T_1>0$. We show that, with $\epsilon_0$ small enough, we can ensure
$T_0 \geq \tau_0\epsilon^{-3}$ and $T_1 \geq \tau_1\epsilon^{-3}$, together with the approximation
\eqref{bound-final-time}.

Let us use the decomposition
\begin{equation}
\label{decomposition-time}
w_n(t) = W(\epsilon (n-t), \epsilon^3 t) + \mathcal{W}_n(t), \quad
p_n(t) = P_{\epsilon}(\epsilon (n-t),\epsilon^3 t) + \mathcal{P}_n(t), \quad n \in \mathbb{Z},
\end{equation}
where $W(\xi,\tau)$ is the considered smooth solution to the log--KdV equation (\ref{LogKdV-background})
(and thus $W$ is $\epsilon$-independent),
whereas the $\epsilon$-dependent function $P_{\epsilon}(\xi,\tau)$ is found from the truncation
of the first equation of the system (\ref{FPU-lattice}) rewritten as
\begin{equation} \label{FPU-lattice-first-trucated}
P_{\epsilon}(\xi + \epsilon,\tau) - P_{\epsilon}(\xi,\tau) =
-\epsilon \partial_{\xi} W(\xi,\tau) + \epsilon^3 \partial_{\tau} W(\xi,\tau).
\end{equation}

We look for an approximate solution $P_{\epsilon}$ to this equation, under the form
\begin{equation}
\label{expansion-P}
P_{\epsilon} := P^{(0)} + \epsilon P^{(1)} + \epsilon^2 P^{(2)} + \epsilon^3 P^{(3)},
\end{equation}
with functions $P^{(j)}$ decaying to zero as $\xi$ goes to infinity.
Plug this ansatz into (\ref{FPU-lattice-first-trucated}) and collect together the powers of $\epsilon$:
{\small \begin{eqnarray*}
\mathcal{O}(\epsilon) :
& \quad \partial_{\xi} P^{(0)} = - \partial_{\xi} W , \quad & \mbox{satisfied when} \quad P^{(0)} = - W, \\
\mathcal{O}(\epsilon^2) :
& \quad \partial_{\xi} P^{(1)} + \frac{1}{2} \partial_{\xi}^2 P^{(0)}= 0 , \quad
& \mbox{satisfied when} \quad P^{(1)} = \frac{1}{2} \partial_{\xi} W, \\
\mathcal{O}(\epsilon^3) :
& \begin{array}[t]{l} \partial_{\xi} P^{(2)} + \frac{1}{2} \partial_{\xi}^2 P^{(1)}
+ \frac{1}{6} \partial_{\xi}^3 P^{(0)} = \\
-\frac{1}{24} \partial_{\xi}^3 W -\frac{1}{2} \partial_{\xi} g(W) ,\end{array}
& \mbox{satisfied when} \quad P^{(2)} = -\frac{1}{8} \partial_{\xi}^2 W - \frac{1}{2} g(W),  \\
\mathcal{O}(\epsilon^4) :
& \partial_{\xi} P^{(3)}
\begin{array}[t]{l} + \frac{1}{2} \partial_{\xi}^2 P^{(2)} \\
+ \frac{1}{6} \partial_{\xi}^3 P^{(1)} + \frac{1}{24} \partial_{\xi}^4 P^{(0)} = 0 , \end{array}
& \mbox{satisfied when} \quad
P^{(3)} = \frac{1}{48} \partial_{\xi}^3 W
+ \frac{1}{4} \partial_{\xi} g(W),
\end{eqnarray*}}
where $g(w) := (1+w) \log(1+w)$.
Recall from the proof of Theorem \ref{theorem-stationary} that we can write the nonlinear potential
in the perturbed form
$$
\tilde{V}_{\epsilon}'(w) = w + \epsilon^2 g(w) + M_{\epsilon}(w),$$
where
$$
M_{\epsilon}(w) = \log^2(1 + w) \int_0^{\epsilon^2} \left( \int_0^x (1 + w)^{1+y} dy \right) dx.
$$
Substituting the decomposition (\ref{decomposition-time}) into the FPU lattice equations (\ref{FPU-lattice}),
we obtain the evolution problem for the error terms
\begin{eqnarray}
\left\{
\begin{split}
& \dot{\mathcal{W}}_n(t) = \mathcal{P}_{n+1}(t) - \mathcal{P}_n(t) + {\rm Res}_n^{(1)}(t), \\
& \dot{\mathcal{P}}_n(t) = \mathcal{W}_n(t) - \mathcal{W}_{n-1}(t) \\
& \qquad \quad + \epsilon^2 g'(W(\epsilon(n-t),\epsilon^3t)) \mathcal{W}_{n}(t)
- \epsilon^2 g'(W(\epsilon(n-1-t),\epsilon^3t)) \mathcal{W}_{n-1}(t) \\
& \qquad \quad + \mathcal{R}_n(W,\mathcal{W})(t) + {\rm Res}_n^{(2)}(t),
\end{split} \right.
\label{FPU-lattice-time}
\end{eqnarray}
where
{\small \begin{equation*}
\begin{split}
\mathcal{R}_n(W,\mathcal{W}) & :=
\epsilon^2 \left( g(W(\epsilon(n-\cdot),\epsilon^3\cdot)+\mathcal{W}_n)
- g(W(\epsilon(n-\cdot),\epsilon^3\cdot))
- g'(W(\epsilon(n-\cdot),\epsilon^3\cdot))\mathcal{W}_{n} \right) \\
& \phantom{tex} - \epsilon^2 \left( g(W(\epsilon(n-1-\cdot),\epsilon^3\cdot)+\mathcal{W}_{n-1})
- g(W(\epsilon(n-1-\cdot),\epsilon^3\cdot))
- g'(W(\epsilon(n-1-\cdot),\epsilon^3\cdot))\mathcal{W}_{n-1} \right) \\
& \phantom{tex} + M_{\epsilon}(W(\epsilon(n-\cdot),\epsilon^3\cdot)+\mathcal{W}_n)
- M_{\epsilon}(W(\epsilon(n-1-\cdot),\epsilon^3\cdot)+\mathcal{W}_{n-1})
\end{split}
\end{equation*}}
and
\begin{equation*}
\begin{split}
{\rm Res}_n^{(1)}(t) & := P_\epsilon(\epsilon(n+1-t),\epsilon^3t) - P_\epsilon(\epsilon(n-t),\epsilon^3t), \\
{\rm Res}_n^{(2)}(t) & :=
\epsilon \partial_{\xi} P_{\epsilon}(\epsilon(n-t),\epsilon^3t)
- \epsilon^3 \partial_{\tau} P_{\epsilon}(\epsilon(n-t),\epsilon^3t) \\
& \phantom{tex} + W(\epsilon(n-t),\epsilon^3t) - W(\epsilon(n-1-t),\epsilon^3t) \\
& \phantom{tex} + \epsilon^2 g(W(\epsilon(n-t),\epsilon^3t)) - \epsilon^2 g(W(\epsilon(n-1-t),\epsilon^3t)).
\end{split}
\end{equation*}

Lemma~\ref{lemma-residual} below deals with estimating the nonlinear and residual term
of the system (\ref{FPU-lattice-time}). Its proof relies on the following lemma,
which is an improvement of Lemma 3.9 from \cite{SW00}.

\begin{lem} \label{Hs-to-l2}
There exists $C>0$ such that for all $X \in H^1(\R)$ and $\epsilon \in (0,1]$,
\begin{equation*}
\| x \|_{l^2} \leq C \epsilon^{-1/2} \| X \|_{H^1},
\end{equation*}
where $x_n := X(\epsilon n)$, $n \in \mathbb{Z}$. 
\end{lem}

\begin{proof}
We first prove the above inequality when $X$ is in the Schwartz class.
Denote $x_n:=X(\epsilon n)$, and let $\hat{x} : \mathbb{R} \rightarrow \mathbb{C}$
be the $2\pi$-periodic $C^\infty$ function defined by
$$
\hat{x}(\theta) := \sum_{n\in\mathbb{Z}} x_n e^{-in\theta},
$$
so that
$$
x_n = \frac{1}{2\pi} \int_{-\pi}^\pi \hat{x}(\theta) \, e^{in\theta} {\rm d}\theta,
\quad n\in\mathbb{Z}.
$$
On the other hand, by the inverse Fourier transform applied to $X$, we have
\begin{eqnarray*}
x_n & = &  \frac{1}{2\pi} \int_{-\infty}^\infty \hat{X}(k) \, e^{ik\epsilon n} {\rm d}k \\
& = & \frac{1}{2\pi\epsilon} \int_{-\infty}^\infty \hat{X}\left(\frac{p}{\epsilon}\right) \, e^{ipn} {\rm d}p \\
& = & \frac{1}{2\pi\epsilon} \sum_{m\in\mathbb{Z}}
\int_{(2m-1)\pi}^{(2m+1)\pi} \hat{X}\left(\frac{p}{\epsilon}\right) \, e^{ipn} {\rm d}p \\
& = & \frac{1}{2\pi\epsilon} \sum_{m\in\mathbb{Z}}
\int_{-\pi}^{\pi} \hat{X}\left(\frac{\theta+2\pi m}{\epsilon}\right) \, e^{in\theta} {\rm d}\theta .
\end{eqnarray*}
Due to the decay of $\hat{X}$, summation and integral can be inverted.
Then, the $2\pi$-periodic $C^\infty$ function
$\displaystyle\theta\mapsto\frac{1}{\epsilon}\sum_{m\in\mathbb{Z}}\hat{X}\left(\frac{\theta+2\pi m}{\epsilon}\right)$
has the same (inverse) Fourier coefficients as $\hat{x}$, so that they coincide:
$$
\hat{x}(\theta) = \frac{1}{\epsilon}\sum_{m\in\mathbb{Z}}\hat{X}\left(\frac{\theta+2\pi m}{\epsilon}\right), \quad \theta \in \mathbb{R}.
$$
Now, using Parseval's equality, we estimate the $l^2$ norm of $x$,
\begin{eqnarray*}
\| x \|_{l^2}^2
& = & \frac{1}{2\pi\epsilon^2} \int_{-\pi}^\pi
\Big| \sum_{m\in\mathbb{Z}} \hat{X}\left(\frac{\theta+2\pi m}{\epsilon}\right) \Big|^2 {\rm d}\theta \\
& \leq & \frac{1}{2\pi\epsilon^2} \int_{-\pi}^\pi \sum_{m_1,m_2\in\mathbb{Z}}
\left| \hat{X}\left(\frac{\theta+2\pi m_1}{\epsilon}\right) \right| \, \left| \hat{X}\left(\frac{\theta+2\pi m_2}{\epsilon}\right) \right| {\rm d}\theta \\
& \leq & \frac{1}{2\pi\epsilon^2} \sum_{m_1,m_2\in\mathbb{Z}} \int_{-\pi}^\pi
\left| \hat{X}\left(\frac{\theta+2\pi m_1}{\epsilon}\right) \right| \, \left| \hat{X}\left(\frac{\theta+2\pi m_2}{\epsilon}\right) \right| {\rm d}\theta .
\end{eqnarray*}
Inserting the weights $(1+\pi^2m_1^2/\epsilon^2)^{-1}(1+\pi^2m_2^2/\epsilon^2)^{-1}$
and using Cauchy--Scwarz inequality, we get
{\small \begin{eqnarray*}
& & \int_{-\pi}^\pi \left| \hat{X}\left(\frac{\theta+2\pi m_1}{\epsilon}\right) \right| \,
\left| \hat{X}\left(\frac{\theta+2\pi m_2}{\epsilon}\right) \right| {\rm d}\theta \leq
\frac{1}{1+\pi^2m_1^2/\epsilon^2} \frac{1}{1+\pi^2m_2^2/\epsilon^2} \\
& & \times \left( \frac12\int_{-\pi}^\pi (1+\pi^2m_1^2/\epsilon^2)^2
\left| \hat{X}\left(\frac{\theta+2\pi m_1}{\epsilon}\right) \right|^2 {\rm d}\theta +
\frac12\int_{-\pi}^\pi (1+\pi^2m_2^2/\epsilon^2)^2
\left| \hat{X}\left(\frac{\theta+2\pi m_2}{\epsilon}\right) \right|^2 {\rm d}\theta \right).
\end{eqnarray*}
}Summing w.r.t. $m_1$ and $m_2$, the two terms in the right-hand side above result in the same quantity,
so that
$$
\| x \|_{l^2}^2
\leq \frac{1}{2\pi\epsilon^2} \left( \sum_{m_1\in\mathbb{Z}} \frac{1}{1+\pi^2m_1^2/\epsilon^2} \right)
\left( \sum_{m_2\in\mathbb{Z}} (1+\pi^2m_2^2/\epsilon^2)
\int_{-\pi}^\pi \left| \hat{X}\left(\frac{\theta+2\pi m_2}{\epsilon}\right) \right|^2 {\rm d}\theta \right).
$$
For $\epsilon\in(0,1]$, the first term in the product takes values between $1$ and
$\sum_{m\in\mathbb{Z}} (1+\pi^2m^2)^{-1} < \infty$. The second term can be compared
with the $H^1$ norm of $X$:
\begin{eqnarray*}
\| X \|_{H^1}^2
& = & \frac{1}{2\pi} \int_{-\infty}^\infty (1+k^2) \left| \hat{X}(k) \right|^2 {\rm d}k \\
& = & \frac{1}{2 \pi \epsilon} \sum_{m\in\mathbb{Z}} \int_{-\pi}^\pi
(1+(\theta+2\pi m)/\epsilon)^2) \left| \hat{X}\left(\frac{\theta+2\pi m}{\epsilon}\right)) \right|^2 {\rm d}\theta.
\end{eqnarray*}
For any $m\in\mathbb{Z}$, $\theta\in[-\pi,\pi]$, we have $(\theta+2\pi m)^2\geq\pi^2m^2$,
so that the factor $(1+(\theta+2\pi m)/\epsilon)^2$ is bounded from below by $(1+\pi^2m^2/\epsilon^2)$.
This gives the desired inequality (with a constant $C$ equal to
$(\sum_{m\in\mathbb{Z}} (1+\pi^2m^2)^{-1/2}$, for example).

When $X$ belongs to $H^1(\mathbb{R})$, we can consider a sequence $\{ X^{(k)} \}_{k\in\mathbb{N}}$
of functions in the Schwartz class converging to $X$ in $H^1$. For each $\epsilon\in(0,1]$
and $n\in\mathbb{Z}$, $X^{(k)}(\epsilon n)$ tends to $X(\epsilon n)$ as $k$ tends to infinity,
and Fatou's lemma concludes the proof.
\end{proof}

\begin{lem} \label{lemma-residual}
Let $W \in C([-\tau_0,\tau_1],H^s(\mathbb{R}))$ be a solution to the log--KdV equation
\eqref{LogKdV-background}, for an integer $s \geq 6$ and $\tau_0,\tau_1\geq0$.
Assume that there exists $r_W>-1$ such that $W\geq r_W$.
Then, there exists a positive constant $C_W$ such that for all
$t\in[-\tau_0\epsilon^{-3},\tau_1\epsilon^{-3}]$ and $\epsilon\in(0,1]$,
\begin{equation}
\label{bound-rem}
\| {\rm Res}^{(1)}(t) \|_{l^2}
+ \| {\rm Res}^{(2)}(t) \|_{l^2} \leq C_W \epsilon^{9/2}.
\end{equation}
Furthermore, for $\epsilon_0\in(0,1]$ and for all $\epsilon\in(0,\epsilon_0]$, let
$\mathcal{W}^\epsilon \in C([-\tau_0\epsilon^{-3},\tau_1\epsilon^{-3}],l^2(\mathbb{Z}))$ be such that,
for some $r>-1$ and $R>0$ independent of $\epsilon$,
\begin{equation}
\label{boundsCalW}
-1 < r \leq W(\epsilon(n-t),\epsilon^3 t) + \mathcal{W}^\epsilon_n(t) \leq R < \infty,
\quad n \in \mathbb{Z}, \;\; t \in [-\tau_0\epsilon^{-3},\tau_1\epsilon^{-3}].
\end{equation}
Then, there exists a positive constant $C_{r,R,W}$ such that,
for all $\epsilon \in (0,\epsilon_0]$, we have
\begin{equation}
\label{bound-non}
\| \mathcal{R}(W,\mathcal{W}^\epsilon)(t) \|_{l^2} \leq
C_{r,R,W} (\epsilon^2 \| \mathcal{W}^\epsilon(t) \|_{l^2}^2
+ \epsilon^4 \| \mathcal{W}^\epsilon(t) \|_{l^2} + \epsilon^{9/2}),
\quad t \in [-\tau_0\epsilon^{-3},\tau_1\epsilon^{-3}].
\end{equation}
In addition, the constant $C_{r,R,W}$ may be kept with the same value when $\epsilon_0$ is decreased.
\end{lem}

\begin{proof}
To obtain the part of estimate \eqref{bound-rem} concerning ${\rm Res}^{(1)}(t)$,
we use the definition \eqref{expansion-P}
of $P_\epsilon$, the expressions of the $P^{(j)}$'s as linear combinations of derivatives of $W$ and $g(W)$,
and Taylor expansions. The coefficients of $\epsilon^0, \dots, \epsilon^4$ vanish, due to the fact that $W$
is a solution to \eqref{LogKdV-background}. As a consequence,
${\rm Res}^{(1)}(t)$ is then expressed as a sum of
integrals of the form
$$
\epsilon^5 \int_0^1 (1-r)^k \partial_\xi^5 W(\epsilon(n-t+r),\epsilon^3t) {\rm d}r
\quad \mbox{and} \quad
\epsilon^5 \int_0^1 (1-r)^l \partial_\xi^2 g(W)(\epsilon(n-t+r),\epsilon^3t) {\rm d}r,
$$
with $0\leq k\leq4$ and $0\leq l\leq1$. The associated $l^2$ norm is them easily estimated in terms
of $\| W \|_{H^6}$, thanks to Lemma~\ref{Hs-to-l2}.
The proof of the rest of estimate \eqref{bound-rem} concerning ${\rm Res}^{(2)}(t)$ follows the same lines.

To prove (\ref{bound-non}), we recall that for all $r>-1$, there exists $C_r>0$ such that for all $w_1, w_2 \geq r$
and $\epsilon>0$,
\begin{eqnarray}
\label{Lipschitz-M}
|M_{\epsilon}(w_1) - M_{\epsilon}(w_2)| \leq \epsilon^4 C_r \left(1 + \max\{w_1,w_2\}\right)^{1+\epsilon^2}
\max\{w_1,w_2\} |w_1-w_2|.
\end{eqnarray}
Then, using again Taylor expansions, we get
\begin{eqnarray*}
\| \mathcal{R}(W,\mathcal{W})(t) \|_{l^2}
& \leq & C \left( \| g''(W(\epsilon(\cdot - t),\epsilon^3 t)) \|_{L^{\infty}} \| \mathcal{W} \|_{l^2}^2
+ \epsilon^4 \| W(\epsilon(\cdot - t),\epsilon^3 t)) \|_{L^{\infty}} \| \mathcal{W} \|_{l^2} \right. \\
& \phantom{t} & \qquad\qquad \left. + \epsilon^5
 \| (\partial_{\xi} W(\epsilon(\cdot - t),\epsilon^3 t))_{n\in\mathbb{Z}} \|_{l^2} \right),
\end{eqnarray*}
which yields the bound (\ref{bound-non}).
\end{proof}

Thanks to Lemma \ref{lemma-residual}, we complete the proof of Theorem \ref{theorem-justification}
using energy estimates. When $\epsilon_0>0$ is given, we consider for each $\epsilon\in(0,\epsilon_0)$
initial data $(w_{{\rm ini},\epsilon},p_{{\rm ini},\epsilon})$ satisfying the bound
\eqref{bound-initial-time}. Fixing
$$
r:=\frac{r_W-1}{2} \in(-1,r_W) \quad \mbox{\rm and} \quad R: = 2R_W > R_W,
$$
with $\epsilon_0$ small enough, we can define (for each $\epsilon\in(0,\epsilon_0)$) a local-in-time solution $(w,p)$
to the FPU lattice equations \eqref{FPU-lattice}, decomposed according to \eqref{decomposition-time}, and then set
$$
T_0^\star(\epsilon) :=
\sup\left\{T_0\in(0,\tau_0\epsilon^{-3}] : \quad
r \leq W(\epsilon(n-t),\epsilon^3 t) + \mathcal{W}_n(t) \leq R ,
\quad n \in \mathbb{Z}, \; t \in [-T_0,0] \right\} .
$$
and
$$
T_1^\star(\epsilon) :=
\sup\left\{T_1\in(0,\tau_1\epsilon^{-3}] : \quad
r \leq W(\epsilon(n-t),\epsilon^3 t) + \mathcal{W}_n(t) \leq R ,
\quad n \in \mathbb{Z}, \; t \in [0,T_1] \right\} .
$$
We shall prove that for $\epsilon_0$ small enough, we have $T_0^\star(\epsilon)=\tau_0\epsilon^{-3}$
and $T_1^\star(\epsilon)=\tau_1\epsilon^{-3}$.

Let us define the energy-type quantity
\begin{equation}
\mathcal{E}(t) := \frac{1}{2} \sum_{n \in \mathbb{Z}} \left[
\mathcal{P}_n^2(t) + \mathcal{W}_n^2(t) + \epsilon^2 g'(W(\epsilon(n-t),\epsilon^3t)) \mathcal{W}^2_{n}(t)
\right].
\label{energy-type}
\end{equation}
With $\epsilon_0<\min\left(1,\| 2g' \|_{L^\infty(r_W,R_W)}^{-1/2}\right)$,
from the bounds \eqref{boundsW}, we get, for $\epsilon\in(0,\epsilon_0)$,
$$
\| \mathcal{P}(t) \|_{l^2}^2 + \| \mathcal{W}(t) \|_{l^2}^2 \leq 4 \mathcal{E}(t),
\quad t\in(-T_0^\star,T_1^\star).
$$

Taking derivative of $\mathcal{E}$ w.r.t. time, we obtain
\begin{equation*}
\begin{split}
\frac{{\rm d} \mathcal{E}}{{\rm d} t} (t) =
\sum_{n \in \mathbb{Z}}
& \Big[ \mathcal{P}_n(t) \mathcal{R}_n(W,\mathcal{W})(t)
+ \mathcal{P}_n(t) {\rm Res}_n^{(2)}(t) \\
& + \mathcal{W}_{n}(t) [1 + \epsilon^2 g'(W(\epsilon(n-t),\epsilon^3t))]  {\rm Res}_n^{(1)}(t) \\
& + \frac{\epsilon^2}{2} g''(W(\epsilon(n-t),\epsilon^3t)) \mathcal{W}^2_{n}(t)
(-\epsilon \partial_{\xi} + \epsilon^3 \partial_{\tau}) W(\epsilon(n-t),\epsilon^3t) \Big].
\end{split}
\end{equation*}

Then, using Lemma \ref{lemma-residual} and the Cauchy--Schwarz inequality, we estimate
\begin{eqnarray*}
\left| \frac{{\rm d} \mathcal{E}}{{\rm d} t} \right|
& \leq & \| \mathcal{P} \|_{l^2} \| \mathcal{R}(W,\mathcal{W}) \|_{l^2}
+ \| \mathcal{P} \|_{l^2} \| {\rm Res}^{(2)} \|_{l^2} + \,\frac32\, \| \mathcal{W} \|_{l^2}
\left\| {\rm Res}^{(1)} \right\|_{l^2} + \,\epsilon^3 C_W \| \mathcal{W}(t) \|_{l^2}^2 \\
& \leq & C_W \mathcal{E}^{1/2}
\left( \epsilon^{9/2} + \epsilon^3 \mathcal{E}^{1/2} + \epsilon^2 \mathcal{E} \right),
\end{eqnarray*}
with a new constant $C_W$.
Choosing $\mathcal{Q} = \mathcal{E}^{1/2}$, we rewrite the energy balance equation in the form
\begin{equation*}
\left| \frac{{\rm d} \mathcal{Q}}{{\rm d} t} \right|
\leq C_W \left( \epsilon^{9/2} + \epsilon^3 \mathcal{Q} + \epsilon^2 \mathcal{Q}^2 \right).
\end{equation*}
By Gronwall's inequality, we obtain
\begin{equation*}
\mathcal{Q}(t) \leq
(\mathcal{Q}(0) + C_W \epsilon^{9/2} |t| ) \, e^{\epsilon^3 C_W |t|},
\quad t \in (-T_0^\star,T_1^\star).
\end{equation*}
Now, the bound \eqref{bound-initial-time} ensures that $\| (\mathcal{W},\mathcal{P}) \|_{L^2}$
is $\mathcal{O}(\epsilon^{3/2})$ at $t = 0$, so that from the definition \eqref{energy-type} of $\mathcal{E}$, $\mathcal{Q}(0)$
is also $\mathcal{O}(\epsilon^{3/2})$, for $\epsilon_0$ small enough. Thus, we get
\begin{equation} \label{lastGronwall}
\mathcal{Q}(t) \leq
C_W (1 + \max(\tau_0,\tau_1)) \, \epsilon^{3/2} e^{C_W \max(\tau_0,\tau_1)},
\quad t \in (-T_0^\star,T_1^\star).
\end{equation}
Finally, choosing $\epsilon_0$ so that the right-hand side in \eqref{lastGronwall} is so small that
$$
|\mathcal{W}_n(t)| \leq \max\left(\frac{1+r_W}{2},R_W\right)
$$
shows that for all $\epsilon\in(0,\epsilon_0)$,
$T_0^\star(\epsilon)=\tau_0\epsilon^{-3}$ and $T_1^\star(\epsilon)=\tau_1\epsilon^{-3}$.
Theorem \ref{theorem-justification} is proved.

\begin{rem} \label{higherorder}
Using instead of \eqref{decomposition-time} an asymptotic expansion
\begin{eqnarray*}
w_n(t) & = & W(\epsilon (n-t), \epsilon^3 t)
+ \sum_{k=1}^K \epsilon^k W^{(k)}(\epsilon (n-t), \epsilon^3 t) + \mathcal{W}_n(t), \\
p_n(t) & = & \sum_{k=0}^K \epsilon^k P^{(k)}(\epsilon (n-t), \epsilon^3 t) + \mathcal{P}_n(t),
\end{eqnarray*}
at any order $K\in\mathbb{N}$, together with expansion of $\tilde{V}_\epsilon'(w)$ in powers of $\epsilon^2$,
we could improve the approximation \eqref{bound-final-time}, replacing $C_0\epsilon^{3/2}$ by
$C_K \epsilon^{K/2}$, for any $K\in\mathbb{N}$. The approximation time remains $\mathcal{O}(\epsilon^{-3})$
in such an improved approximation.
\end{rem}

\section{Discussion}

The comparison of the two results given by Theorems \ref{theorem-stability} and \ref{theorem-justification}
raises a serious concern on the validity of the KdV-type approximation for the stability theory
of the travelling waves in the FPU lattices. On one hand, Theorem \ref{theorem-stability} yields
nonlinear stability of the FPU travelling waves up to the time scale of $\mathcal{O}(\epsilon^{-3})$
at which the travelling waves are proved to satisfy the specific scaling leading to the
KdV-type approximation. On the other hand, Theorem \ref{theorem-justification} shows that
the nonlinear stability of the FPU travelling waves may depend on the orbital stability of the travelling waves
in the KdV-type equations. It happens for the log--KdV equation (\ref{LogKdV-background}) that the
positive travelling waves are orbitally stable for all amplitudes \cite{H11}. However, it does not have
to be the case for all KdV-type equations.

For instance, if we consider the FPU lattice (\ref{FPU-lattice}) with the nonlinear potential
$$
\tilde{V}_{\epsilon}(w) = \frac{1}{2} w^2 + \frac{\epsilon^2}{p+1} w^{p+1}, \quad \mbox{\rm for an integer } \; p \geq 2,
$$
the results of Theorems \ref{theorem-stability} and \ref{theorem-justification} hold true but the
generalized KdV equation takes the form
\begin{equation}
\label{genKdV}
2 W_{\tau} + \frac{1}{12} W_{\xi \xi \xi} + (W^p)_{\xi} = 0.
\end{equation}
The generalized KdV equation (\ref{genKdV}) is known to have orbitally stable travelling waves
for $p = 2,3,4$ and orbitally unstable travelling waves for $p \geq 5$ \cite{Pava}. Thus, it may first appear
that the results of Theorems \ref{theorem-stability} and \ref{theorem-justification} are in contradiction.

No contradiction arises as a matter of fact. The energy methods used in the proof of
Theorems \ref{theorem-stability} and \ref{theorem-justification} give the upper bounds
on the approximation errors (\ref{bound-final}) and (\ref{bound-final-time}) to be
exponentially growing at the time scale of $\epsilon^3 t$, that is, on the time scale of
$\tau$. The unstable eigenvalues of the linearized generalized KdV equation (\ref{genKdV})
at the travelling waves (if they exist) lead to the exponential divergence at the time
scale of $\tau$, which can not be detected with the approximation results provided by
Theorems \ref{theorem-stability} and \ref{theorem-justification}.

Therefore, within the approximation results of Theorems \ref{theorem-stability} and \ref{theorem-justification},
we are still left wondering if the travelling waves of the FPU lattice with the nonlinear potential
$\tilde{V}_{\eps}$ for $\epsilon > 0$ small enough are nonlinearly stable
at the time scale of $\tau = \epsilon^3 t$.
What the stability result of Theorem \ref{theorem-stability}
rules out is the presence of the unstable eigenvalues of the linearized FPU lattice
of the size $\mathcal{O}(\epsilon^q)$ for any $q < 3$. However, unstable eigenvalues
of the size $\mathcal{O}(\epsilon^q)$ for $q \geq 3$ are still possible.

Note that the result of Theorem \ref{theorem-stability}
does not depend on the nonlinear potential $\tilde{V}_{\eps}$ as long
as the latter provides the specific scaling leading to the
KdV-type approximation. We did not have to
construct the two-dimensional manifold of the travelling waves or use
projections and modulation equations from the theory in \cite{FP2,FP3,FP4}.
Although this theory gives a complete proof of
nonlinear orbital stability of FPU travelling waves of small amplitudes,
it relies on the information about the spectral and asymptotic stabilities
of the KdV travelling waves, which is only available in the case of the integrable
KdV equation (\ref{genKdV}) with $p = 2$ (such information may also
be available in the case $p = 3$, since the corresponding so-called ``modified KdV''
equation is integrable as well). It is not clear at the present
time if any bits of the information needed to proceed with the theory in \cite{FP2,FP3,FP4}
can be obtained for the log--KdV equation (\ref{LogKdV-background}), although
the existing theory in \cite{H11} excludes unstable eigenvalues and guarantees
nonlinear orbital stability of the travelling waves in the log--KdV equation.

\vspace{0.25cm}

{\bf Acknowledgement.} The authors thank G. James for bringing up the problem and
J. H\"{o}wing for pointing out to his work \cite{H11}.
D.P. is supported by the Chaire d'excellence ENSL/UJF. He thanks members of  Institut Fourier, Universit\'e Grenoble
for hospitality and support during his visit (January-June, 2014).

\end{document}